\documentclass{amsart}

\usepackage{AHS-style}
\newcommand{\lo}{{\lambda_0}}

\begin{document}

\title[Chiral algebras, factorization algebras, and $(A,H,S)$-vertex algebras]{Chiral algebras, factorization algebras, and Borcherds's ``singular commutative rings'' approach to vertex algebras}
\author{Emily Cliff}
\date{November 2019}
\begin{abstract}
We recall Borcherds's approach to vertex algebras via ``singular commutative rings'', and introduce new examples of his constructions which we compare to vertex algebras, chiral algebras, and factorization algebras. We show that all vertex algebras (resp. chiral algebras or equivalently factorization algebras) can be realized in these new categories $\VA(A,H,S)$, but we also show that the functors from $\VA(A,H,S)$ to vertex algebras or chiral algebras are not equivalences: a single vertex or chiral algebra may have non-equivalent realizations as an $(A, H,S)$-vertex algebra.
\end{abstract}
%
%
%
\maketitle

\tableofcontents

\section{Introduction}
A vertex algebra describes the symmetries of a two-dimensional conformal field theory, while a factorization algebra over a complex curve consists of local data in such a field theory. Roughly, the factorization structure encodes collisions between local operators. The two perspectives are equivalent, in the sense that, given a fixed open affine curve $X$, the category of vertex algebras over $X$ is equivalent to the category of factorization algebras over $X$ \cite{HL, BD}. In 1999, Borcherds \cite{BorQVA} introduced yet another approach to studying vertex algebras: given some input data $(A, H, S)$, he introduces a category equipped with a ``singular'' tensor product, and he defines $(A, H, S)$-vertex algebras to be commutative ring objects in this category. For a particular choice of input data, $(A, H, S_B)$, he proves that these singular ring objects can be used to produce examples of vertex algebras. An advantage of this approach is that once one understands the definition of the category, one needs to work only with ring axioms rather than with the less familiar axioms of vertex algebras. However, at the time of Borcherds's paper, it was not clear whether all vertex algebras could be realized via this approach. In this paper, we prove that they can, if we modify the input data appropriately. More generally, we address the question of how closely the categories of $(A, H, S)$-vertex algebras and ordinary vertex algebras (or factorization algebras, or equivalently chiral algebras) are related. 

\begin{thm}\label{Theorem: Borcherds doesn't give an equivalence}[Theorem \ref{Theorem: Borcherds doesn't give an equivalence, actual}]
The functor $\Phi$ from the category $\VA(A, H, S_B)$ of $(A,H,S_B)$-vertex algebras to the category $\VA$ of ordinary vertex algebras (arising from Borcherds's constructions) is not an equivalence. 
\end{thm}

While this is a negative result, we can also obtain positive results by modifying our input data. 

\begin{thm}\label{Theorem: main results on X}[Theorems \ref{Theorem: AHS to vertex algebra}, \ref{Theorem: AHS to chiral algebra}, \ref{Theorem: Factorization to AHS}, \ref{Theorem: composition}, \ref{Theorem: Borcherds doesn't give an equivalence, actual}]
Given $X$ an open affine subset of $\BBA^1_{\BC}$, there exist input data $(A, H, S_X)$ such that we have 
\begin{itemize}
\item a functor $\Phi_X$ from the category $\VA(A, H, S_X)$ of $(A, H, S_X)$-vertex algebras to the category $\VA(X)$ of vertex algebras on $X$, and
\item a functor $\Gamma_X$ from the category $\FA(X)$ of factorization algebras on $X$ to the category $\VA(A, H, S_X)$, 
\end{itemize}
such that $\Gamma_X$ is a section (or left-inverse) of $\Phi_X$ in the following sense. Given a vertex algebra $V$ over $X$, let $\CB^V$ denote the corresponding factorization algebra over $X$. Then $\Phi_X \circ \Gamma_X (\CB^V) =V$. In particular, $\Phi_X$ is essentially surjective. 

However, $\Phi_X$ fails to be an equivalence: distinct $(A, H, S_X)$-vertex algebras may give rise to isomorphic vertex algebras over $X$. 

Furthermore, the composition of the functor $\Phi_X$ with the well-studied equivalence $\Psi_X$ (of \cite{BD}) from $\VA(X)$ to the category $\CAlg(X)$ of chiral algebras on $X$ can be described explicitly, without reference to the intermediate category $\VA(X)$. 
\end{thm}

The functors and results of Theorem \ref{Theorem: main results on X} extend to the translation-equivariant setting, which allows us to use the equivalence $\Xi^{\VA}$ between $\VA$ and the category $\VA(\BBA^1)^{\BBA^1}$ of translation-equivariant vertex algebras on $X=\BBA^1$ to relate the new functor $\Phi_{\BBA^1}$ to the functor $\Phi$ resulting from Borcherds's original constructions.

\begin{thm}\label{Theorem: translation-equivariance}[Proposition \ref{Prop: Sb to translation equivariant by tensoring}]
There is a functor $\Xi$ from $\VA(A,H, S_B)$-vertex algebras to the category $\VA(A, H, S_{\BBA^1})^{\BBA^1}$ of translation-equivariant $(A, H, S_{\BBA^1})$-vertex algebras, which intertwines the functors $\Phi$ and $\Phi_{\BBA^1}$ via $\Xi^{\VA}$. 
\end{thm}

The relationships between the categories in question can thus be summarized in the following diagram (where $\CAlg(\BBA^1)^{\BBA^1}$, respectively $\FA(\BBA^1)^{\BBA^1}$, denotes the category of translation-equivariant chiral algebras, respectively factorization algebras, on $\BBA^1$):

\begin{center}
\begin{tikzpicture}[>=angle 90, bij/.style={above,inner sep=0.5pt}]
\matrix(b)[matrix of math nodes, row sep=2em, column sep=2em, text height=1.5ex, text depth=0.25ex]
{\VA & \VA(\BBA^1)^{\BBA^1} & \CAlg(\BBA^1)^{\BBA^1} & \FA(\BBA^1)^{\BBA^1}.\\
 \VA(A, H, S_B) & \VA(A, H, S_{\BBA^1})^{\BBA^1} & & \\};
 \path[->, font=\scriptsize]
 (b-1-1) edge node[above, bij]{$\sim$} node[below]{$\Xi^{\VA}$} (b-1-2)
 (b-1-2) edge node[above, bij]{$\sim$} (b-1-3)
 (b-1-3) edge node[above, bij]{$\sim$} (b-1-4)
 (b-2-1) edge node[below]{$\Xi$} (b-2-2)
 (b-2-1) edge node[left]{$\Phi$} (b-1-1)
 (b-2-2) edge node[right]{$\Phi_{\BBA^1}$} (b-1-2)
 (b-1-4) edge[dashed, bend left=20] node[below right]{$\Gamma_{\BBA^1}$}(b-2-2);
\end{tikzpicture}
\end{center}

Here the arrow corresponding to $\Gamma_{\BBA^1}$ is dashed to remind us that the morphisms in the right-hand part of the diagram are not completely compatible: one does not obtain the identity functor if one begins tracing a clockwise cycle at the bottom vertex $\VA(A, H, S_{\BBA^1})^{\BBA^1}$, but only if one begins at one of the vertices in the top row. 

\subsection{Future directions/applications}
In the preprint \cite{Joyce}, Joyce defines vertex algebra structures on homology of moduli spaces using techniques similar to Borcherds's construction of the lattice vertex algebra via the functor $\Phi$. It is exciting that these structures exist, but it is difficult to see where the geometry of the underlying moduli spaces is used in these constructions. The present paper illuminates the relationship between $(A, H, S)$-vertex algebras and factorization algebras, which are more explicitly geometric objects; in future work we will attempt to exploit this relationship to better understand Joyce's results. 

As suggested by its title, the paper \cite{BorQVA} generalized the definition of $(A, H, S)$-vertex algebras to $(A, H, S)$-\emph{quantum vertex algebras}. There are several approaches to quantum vertex algebras, but Borcherds's approach seems particularly promising, particularly in light of the results of the current paper. For example, Anguelova--Bergvelt \cite{AB} compare several approaches, and adopt Borchderds's techniques for constructing examples of quantum vertex algebras, but give as their reason for not framing their axioms in the language of $(A,H,S)$-vertex algebras the observation that \emph{``even for
classical vertex algebras it seems not known how to include such basic examples as affine vertex algebras in the $(A, H, S)$-framework''}. However, we now see that the functor $\Gamma_{\BBA^1}$ allows us to view any vertex algebra as an $(A,H,S)$-vertex algebra. In future work, we will adapt the axioms of quantum $(A,H,S)$-vertex algebras to the setting of factorization algebras in a way compatible with the suitable analogue of $\Gamma_{\BBA^1}$, and will study the resulting objects. 

\subsection{Notation and conventions} \label{Subsec: notation}
For simplicity, we work over $\BC$ throughout (unless explicitly stated otherwise), although some of the definitions and results admit generalizations to other fields and even rings. (In particular, in \cite{BorQVA}, Borcherds gives the definition of an $(A, H, S)$-vertex algebra over an arbitrary ring.)

We work with $X$ an open affine subset of $\BBA^1$, and we fix a choice of global coordinate $x$ on $X$. This allows us to express $\Gamma(X, \CO_X)$ as a localization of $\BC[x]$; it also means that we have a derivation $\frac{d}{dx}$ on $\Gamma(X, \CO_X)$. 

\subsection{Structure of the paper}
The paper is organized as follows. In Section \ref{Section: Background, general} we provide background on vertex algebras, chiral algebras, factorization algebras, and the relationships between them, while in Section \ref{Section: Background, AHS} we review the definitions of $(A, H, S)$-vertex algebras, following \cite{BorQVA}. We also introduce the variations $(A, H, S_X)$ of input data that will be needed for our applications. In Section \ref{Section: from AHS to vertex algebras} we construct the functor $\Phi_X$ from $\VA(A, H, S_X)$ to $\VA(X)$, and in \ref{Section: from AHS to chiral} we give an explicit description of the composition of $\Phi_X$ with the well-studied morphism $\Psi_X$ from $\VA(X)$ to $\CAlg(X)$. In Section \ref{Section: from factorization to AHS} we construct the functor $\Gamma_X$ from $\FA(X)$ to $\VA(A, H, S_X)$, and in Section \ref{sec: composition of functors} we prove that $\Gamma_X$ gives a Section of $\Phi_X$ in the sense described in Theorem \ref{Theorem: main results on X}. We study the translation-equivariant version of the story in Section \ref{Section: translation-equivariant}. Finally, in Section \ref{sec: failure to be an equivalence} we show that $\Phi$ cannot be an equivalence, by studying the specific example of the Virasoro vertex algebra mapping to the rank-one lattice vertex algebra. 

\subsection{Acknowledgements}
I thank Dominic Joyce for introducing me to the paper \cite{BorQVA}, and for providing, via his work \cite{Joyce} on vertex algebra structures on the homology of moduli spaces, the motivation to study the relationship between Borcherds's constructions and factorization algebras. Thanks also to Kobi Kremnizer, Thomas Nevins, and Reimundo Heluani for helpful discussions, and to Yi-Zhi Huang for useful email communications. I am grateful as well to Anna Romanov for comments on a draft of the paper. 

\section{Background: Vertex algebras, chiral algebras, and factorization algebras}\label{Section: Background, general}

In this section, we recall the definitions of the key players: vertex algebras, chiral algebras, and factorization algebras over $X$. Although the definitions of chiral algebras and factorization algebras make sense for more general $X$, we restrict ourselves to $X$ open affine subsets of $\BBA^1$, which is necessary for the definition of a vertex algebra over $X$. As in \ref{Subsec: notation}, we fix a global coordinate $x$ on $X$, which determines a derivation $\frac {d}{dx}$ on $\Gamma(X, \CO_X)$. 

\begin{defn}\label{Def: vertex algebra}
A \emph{vertex algebra} 
\begin{align*}
V=(V, \vac, T, Y(\bullet, z))
\end{align*}
consists of the following data:
\begin{itemize}
\item a vector space $V$ over the field $k$;
\item a distinguished element $\vac \in V$, called the \emph{vacuum element};
\item a linear map $T: V \to V$ called the \emph{translation operator};
\item a linear map $Y(\cdot, z): V \otimes V \to V[\![z^{\pm 1} ]\!]$. We write
\begin{align*}
Y(v, z)u = \sum _{n \in \BZ} v_n (u) z^{-n-1},
\end{align*}
where for each $n \in \BZ$, $v_n$ is an endomorphism of $V$. 
\end{itemize}
These data are subject to the following axioms:
\begin{description}
\item[The lower truncation condition] For fixed $v, u \in V$, there exists $N \gg 0$ such that for all $n \ge N$, $v_n u = 0$. 
\item[The vacuum condition] $Y(\vac, z) = \Id_V$. 
\item[The creation property] $Y(v, z)\vac = v + \sum_{n \le -2} v_n\vac z^{-n-1}$.
\item[The derivative property] $\frac{d}{dz} Y(v,z) = Y(Tv, z)$.
\item[The locality condition] This says that for any two $v,u \in V$, the corresponding fields $Y(v, z)$ and $Y(u, w)$ are \emph{mutually local}. More explicitly, for any third vector $a \in V$, there exists some $N \gg 0$ such that 
\begin{align*}
(z-w)^N Y(v, z) Y(u, w) a = (z-w)^N Y(u, w) Y (v, z) a \in V[\![ z^{\pm1}, w^{\pm 1}]\!].
\end{align*}
\end{description}
A \emph{morphism} of vertex algebras is a linear map preserving all of the above structure. 
\end{defn}

\begin{eg}\label{Example: lattice VOA}
An important example for us will be the so-called \emph{lattice vertex algebra}, $V_L$. Let us fix some notation for this example, which will be necessary for computations later on; we follow the notation of Section 5.4 of \cite{FBZ}. Fix $L$ an even integral lattice, with $(\cdot, \cdot): L \times L \to \BZ$ its positive definite symmetric bilinear form. We let $\fh = L \otimes_\BZ \BC$, viewed as a commutative Lie algebra, and we denote by $\widehat{\fh}$ the \emph{Heisenberg Lie algebra} associated to $L$, a central extension of $\fh((t))$ by $\BC \mathbb{1}$. We consider the Weyl algebra $\widetilde{\CH}_L$, which is the (completed) enveloping algebra of $\widehat{\fh}$ modulo the relation $\mathbb{1} = 1$. It has (topological) generators $h_n$ for $h \in \fh$ and $n \in \BZ$, satisfying $[h_n, g_m] = n (h, g) \delta_{n, -m}$. 

Now for any $\lambda \in L$ we have the \emph{Fock representation} $\pi_\lambda$ of $\widetilde{\CH}_L$, which is generated by a single vector $| \lambda \rangle$ subject to the relations
\begin{align*}
h_n | \lambda \rangle = 0 \text{ for } n > 0; \qquad h_0 | \lambda \rangle =(\lambda, h) | \lambda \rangle. 
\end{align*}

With this data, we can define the lattice vertex algebra $V_L$. As a vector space, we have 
\begin{align*}
V_L = \bigoplus_{\lambda \in L} \pi_\lambda.
\end{align*}

It has a natural structure of vertex algebra (depending on a choice of 2-cocycle $c: L \times L \to \{\pm 1\}$), which we do not spell out here, except to say that the translation operator $T$ is given by acting by 
\begin{align*}
\frac{1}{2}\sum_{a \in A} \sum_{n \in \BZ} (\lambda_a)_n (\lambda^a)_{-n-1},
\end{align*}
where $\{\lambda_a\}_{a \in A}$ is a basis of $L$, and $\{\lambda^a\}_{a \in A}$ gives a dual basis in $\fh = L \otimes_\BZ \BC$. More details can be found in Section 5.4 of \cite{FBZ}.
\end{eg}

\begin{eg}\label{Example: Virasoro vertex algebra}
Another important example is the \emph{Virasoro vertex algebra} $\Vir_c$ at level $c \in \BC$. Consider the \emph{Virasoro Lie algebra} $\Vir$, which is a central extension of the Lie algebra $\Der \BC [\![ t ]\!]$ of derivations of the ring $\BC [\![ t ]\!]$. It has topological generators $L_n = -t^{n+1}\partial_t, n \in \BZ$, and $\bc$, satisfying to the following relations:
\begin{align*}
[L_n, L_m] = (n-m)L^{n+m} + \frac{n^3 - n}{12} \delta_{n, -m} \bc \quad \forall n, m; \qquad [L_n, \bc] = 0 \quad \forall n.   
\end{align*}

Given a complex number $c \in \BC$, the representation $\Vir_c$ is generated by a single vector $v_c$, subject to the relations:
\begin{align*}
L_n v_c = 0 \text{ for } n \ge -1; \qquad \bc v_c = c v_c.
\end{align*}

The Virasoro vertex algebra encodes the conformal symmetry of a vertex algebra corresponding to a 2d conformal field theory. More precisely, a \emph{conformal structure} (of central charge $c$) on a vertex algebra $V$ is a morphism of vertex algebras 
\begin{align*}
\phi: \Vir_c \to V.
\end{align*}
It is not difficult to check that the data of such a map is determined by specifying the single vector $\phi(L_{-2}v_c) \in V$, which is then called the \emph{conformal vector} and is often denoted by $\omega$. This vector must satisfy certain properties in order for the morphism $\phi$ to be well-defined (see for example Section 2.5 of \cite{FBZ}). 
\end{eg}

\begin{eg}\label{Example: conformal structure on lattice VOA}
For example, in the case of a rank 1 lattice $L = \sqrt{N} \BZ$ (for $N$ even), there is an infinite family $\{\phi_\lambda\}$ of conformal structures on $V_L$, parametrized by a choice of $\lambda \in \frac{\sqrt{N}}{2} \BZ$. Let us denote by $b$ the basis element $\sqrt{N} \otimes \frac{1}{\sqrt{N}}$ of $\fh = L \otimes_\BZ \BC$. For a fixed choice of $\lambda \in  \frac{\sqrt{N}}{2} \BZ$, the central charge will be $c_\lambda = 1 - 12\lambda^2$, and we have 
\begin{align}\label{Eq: conformal vector}
\omega_\lambda \defeq \phi_\lambda(L_{-2}v_{c_\lambda}) = \frac{1}{2} b_{-1}^2\vac + \lambda b_{-2}\vac.
\end{align}

\end{eg}

\begin{defn}[Def. 5.1, \cite{HL}]
A \emph{vertex algebra over $X$} consists of a vertex algebra $(V, \vac, T, Y(\cdot, z))$ together with the structure of an $\Gamma(X, \CO_X)$-module on $V$ such that
\begin{itemize}
\item For any $f, g \in \Gamma(X, \CO_X)$, and any $v, u \in V$, we have
\begin{align*}
Y(f(x)\cdot v, z) (g(y) \cdot u) = f(x + z) g(x) \cdot Y(v, z)u,
\end{align*}
where $x$ and $y$ are the global coordinates of the first and second copies of $X$ respectively, acting on $V \otimes V$;
\item For any $f \in \Gamma(X, \CO_X)$ and any $v \in V$ we have
\begin{align*}
T(f(x)\cdot v) = (\frac{d}{dx} f(x)) \cdot v + f(x) \cdot (T v).
\end{align*}
\end{itemize}
\end{defn}

We will let $\VA$ denote the category of vertex algebras, and $\VA(X)$ denote the category of vertex algebras over $X$. 

\begin{defn}[Sec. 3.3, \cite{BD}]
A \emph{non-unital chiral algebra} over $X$ consists of a right $\CD_X$-module $\CA$ equipped with a morphism of $\CD$-modules
\begin{align*}
\mu^\ch : j_* j^* \CA \boxtimes \CA \to \Delta_! \CA
\end{align*}
on $X^2$, which satisfies skew-symmetry and the Jacobi identity. This morphism is called the \emph{chiral bracket}. 

Here $\Delta: X \to X^2$ is the diagonal embedding, and $j$ is the open embedding of its complement $U$.
\end{defn}

\begin{rmk}\label{Remark: affine chiral algebras}
Because $X$ is affine, the data of the $\CD_X$-module $\CA$ is determined by its global sections $A$, viewed as a (right) module over $\CD(X) = \Gamma(X, \CD_X)$.

 Similarly, because $X^2$ and $U$ are also affine, the data of $\mu^\ch$ can also be described in terms of its action on global sections; for this we adopt some notation that will be convenient later on. For a finite set $I$, let $X^I$ denote the product of $|I|$ copies of $X$, let $x_i$ be the fixed global coordinate on the $i$th factor, and let $S_X(I) = \Gamma(X^I, \CO_{X^I})$. In particular, by a slight abuse of notation, we denote by $S_X(1)$ the coordinate ring $\Gamma(X, \CO_X)$, and by $S_X(1,2)$ the global sections $\Gamma(X^2, \CO_{X^2})$. Furthermore, let us denote by $S_X(1:2)$ the sections $\Gamma(U, \CO_{X^2})$, so that $S_X(1:2)$ is the localization of $S_X(1,2)$ by $(x_1 - x_2)$. With this notation in mind, we see that
\begin{align*}
\Gamma(X^2, j_* j^* \CA \boxtimes \CA) = (A \otimes_\BC A) \otimes_{S_X(1,2)} S_X(1:2).
\end{align*}
On the other hand, identifying $\Delta_! \CA = \frac{j_* j^* (\Omega_X \boxtimes \CA)}{\Omega_X \boxtimes \CA}$ as in Section 19.1.1 of \cite{FBZ}, we have
\begin{align*}
\Gamma(X^2, \Delta_! \CA) = \frac{(S_X(1)dx_1 \otimes_\BC A) \otimes_{S_X(1,2)} S_X(1:2)} {S_X(1)dx_1 \otimes_\BC A}.
\end{align*}
In particular, this module is spanned over $\BC$ by elements that look like 
\begin{align*}
(f(x_1)dx_1 \otimes a) \otimes (x_1 - x_2)^{-N} \text{ mod }  S_X(1)dx_1 \otimes_\BC A,
\end{align*}
for $f(x_1) \in S_X(1)$, $a \in A$, and $N >0$. Then the chiral bracket $\mu^\ch$ is determined by a morphism of these right $\CD(X^2)$-modules satisfying properties corresponding to the skew-symmetry and the Jacobi identity.
\end{rmk}

\begin{eg}\label{Example: unit chiral algebra}
It is straightforward to check that the natural map $j_* j^* \omega_X \boxtimes \omega_X \to \Delta_! \omega_X$ makes the canonical bundle $\omega_X = \Omega_X$ into a chiral algebra. 
\end{eg}

\begin{defn}[Sec. 3.3.3, \cite{BD}]
A \emph{unital} chiral algebra $\CA$ is a chiral algebra $(\CA, \mu^\ch)$ equipped with a map
\begin{align*}
u: \omega_X \to \CA
\end{align*}
of chiral algebras such that the restriction of the chiral bracket $\mu^\ch$ to $j_*j^*(\omega_X \boxtimes \CA)$ is the canonical quotient map $j_*j^*(\omega_X \boxtimes \CA)  \surj \Delta_! \CA$. 

From now on, all chiral algebras will be assumed to be unital unless explicitly stated otherwise. We let $\CAlg(X)$ denote the category of (unital) chiral algebras on $X$. 
\end{defn}

Before we can make the definition of a factorization algebra, we need to set some additional notation.
\begin{notation}\label{Notation: categories of finite sets}
Let $\Fin$ denote the category of all finite sets with maps between them. On the other hand, let $\fSet$ denote the category of all \emph{non-empty} finite sets with morphisms being only the surjections between them. (Note that later on we will have a third category $\FinNEq$, but let us postpone its definition until Notation \ref{Notation: finite sets with equivalence relation} when it is actually needed.) 

A surjection $\alpha: I \surj J$ determines a closed embedding $\Delta(\alpha): X^J \to X^I$. On the other hand, an arbitrary map $\alpha: I \to J$ factors into a surjection following by an injection:
\begin{align*}
I \xrightarrow{\alpha^\prime} \im(\alpha) \defeq K \xrightarrow{\beta} J.
\end{align*}
The injection $\beta: K \to J$ determines a projection $\pi(\beta): X^J \to X^K$. We denote the composition of $\pi(\beta)$ and $\Delta(\alpha^\prime)$ by $\phi(\beta): X^J \to X^I$. 

The map $\alpha$ also determines an open embedding $U(\alpha) \emb X^I$, where
\begin{align*}
U(\alpha) = \{ (x_i) \in X^I \ \vert \ x_{i_1} \ne x_{i_2} \text{ unless } \alpha(i_1) = \alpha(i_2)\}. 
\end{align*}
We denote this open embedding by $j(\alpha)$. 
\end{notation}

\begin{defn}[Section 3.4, \cite{BD}]\label{Def: factorization algebra}
A \emph{factorization algebra} on $X$ consists of a family $\left\{\CB_{X^I} \in \CD(X^I) \right\}_{I \in \fSet}$ of left $\CD$-modules together with isomorphisms 
\begin{align*}
\nu_\alpha&: \CB_{X^I} \EquivTo \Delta(\alpha)^* \CB_{X^J}; \\
d_\alpha&:  j(\alpha)^* \left(\boxtimes_{j \in J} \CB_{X^{I_j}} \right) \EquivTo j(\alpha)^* \CB_{X^I}
\end{align*}
for any $\alpha: I \surj J \in \fSet$. We require that the $\CB_{X^I}$ have no non-zero local sections supported at the diagonal divisor. We also require that the $\nu_\alpha$ and $d_\alpha$ satisfy compatibility conditions with respect to compositions of maps $\alpha$. 
\end{defn}

\begin{eg}
It is a straightforward exercise to check that the collection $\CO = \left\{\CO_{X^I}\right\}$ can be given a natural structure of factorization algebra. 
\end{eg}

\begin{defn}[Section 3.4.4, \cite{BD}]
A factorization algebra $\CB = \{\CB_{X^I}\}$ on $X$ is \emph{unital} if it is equipped with a map $u: \CO \to \CB$ of factorization algebras which satisfies the following two conditions:
\begin{enumerate}
\item Let $1 \in \CB_X$ denote the image of the unit of $\CO_X$ under $u$. For any local section $b$ of $\CB_{X}$, $1 \otimes b$ is a local section of $\CB_X \boxtimes \CB_X$, so we can view it as a section of $j_* j^* (\CB_X \boxtimes \CB_X)$; hence $d_{\id}(1 \otimes b)$ is a section of $j_* j^* (\CB_{X^2})$. We require that it is is actually a section of $\CB_{X^2}$. By abuse of notation, we will write $1 \otimes b \in \CB_{X^2}$. 
\item We require furthermore that $\Delta^*(1 \otimes b) = b$, under the identification of $\Delta^*(\CB_{X^2})$ with $\CB_X$. 
\end{enumerate} 
\end{defn}

\begin{rmk}[Remark on the definition]
Beilinson and Drinfeld (Section 3.4.5 \cite{BD}) remark that a unital factorization algebra is equivalent to the following data: 
\begin{itemize}
\item For any $I \in \Fin$ (in particular for $I = \emptyset$) we have a left $\CD$-module on $X^I$. 
\item For any morphism of finite sets $\alpha: I \to J$ we have morphisms of left $\CD$-modules
\begin{align*}
\nu_\alpha&: \phi(\alpha)^* \CB_{X^I} \to \CB_{X^J}; \\
d_\alpha&:  j(\alpha)^* \left(\boxtimes_{j \in J} \CB_{X^{I_j}} \right) \to j(\alpha)^* \CB_{X^I}.
\end{align*}
\end{itemize}
These are required to satisfy the same compatibilities under compositions of the maps $\alpha$ as before; however, we require them to be isomorphisms only when $\alpha$ is surjective. We again require that the sheaves $\CB_{X^I}$ have no non-zero local sections supported on the diagonal divisor. Moreover, $\CB_{X^\emptyset}$ must be non-zero. 
\end{rmk}

From now on, all factorization algebras will be assumed to be unital, unless specified otherwise. Let $\FA(X)$ denote the category of (unital) factorization algebras on $X$. 

Let us summarize the relationships between these categories. 

\begin{thm} \label{Theorem: everybody else's theorem}
Let $X \subset \BBA^1$ be an open affine subset of $\BBA^1$, with a fixed global coordinate $x$. 
\begin{enumerate}
\item (Thm. 5.4 \cite{HL}.) The category of vertex algebras over $X$ is equivalent to the category of chiral algebras over $X$. 
\item The category of vertex algebras is equivalent to the category of translation-equivariant vertex algebras on $\BBA^1$, and hence to the category of translation-equivariant chiral algebras on $\BBA^1$.
\item (Thm. 3.4.9 \cite{BD}.) The category of chiral algebras on $X$ is equivalent to the category of factorization algebras on $X$. 
\end{enumerate}
\end{thm}

\begin{rmk}
\begin{enumerate}
\item Many of ideas in the proof of part (1) appear also in Section 19.2 of \cite{FBZ}.
\item Part (2) can be proved as a corollary of part (1), and has widely been accepted as true throughout the literature. However, an explicit description of the functors appears in the appendix of the recent paper \cite{BDHK}.
\item Part (3) is proved in greater generality (for schemes $X$ of arbitrary dimension $n$) in \cite{FG}; however, we do not need this level of generality. 
\item We will not comment on the proofs here---however, at certain stages in arguments in this paper, we will describe some of the equivalences above more explicitly, as needed. 
\end{enumerate}
\end{rmk}

\section{Background: Borcherds's category of \texorpdfstring{$(A, H, S)$}{(A,H,S)}-vertex algebras}\label{Section: Background, AHS}
In this section, we introduce the main objects of Borcherds's paper \cite{BorQVA}. Borcherds works in an arbitrary symmetric monoidal category $A$; for simplicity, we will restrict our attention to the case that $A = \Vect$, the category of vector spaces over $\BC$. We will also introduce new examples, along with geometric perspectives which will be useful in relating Borcherds's constructions to more geometric objects (chiral algebras and factorization algebras) in subsequent sections. 

\begin{notation}\label{Notation: finite sets with equivalence relation}
Recall from \ref{Notation: categories of finite sets} the category $\Fin$ of finite sets with arbitrary maps between them. We are also interested in the category $\FinNEq$ whose objects are finite sets equipped with an equivalence relations, and whose morphisms are functions between the finite sets which preserve \emph{inequivalence}: i.e., we consider only $\alpha: I \to J$ with the property that $\alpha(i_1) \sim \alpha(i_2)$ only if $i_1 \sim i_2$. 

We will consider $\Fin$ as a full subcategory of $\FinNEq$ by letting an ordinary finite set $I$ carry the degenerate equivalence relation, so all elements of $I$ are equivalent. 

Note that giving an equivalence relation on a finite set $I$ is equivalent to giving a surjection from $I$ to another set: we will sometimes write $\pi_I: I \surj I^\prime$ for this surjection, and identify $I^\prime$ with the set of equivalence classes of $I$. In the notation of \ref{Notation: categories of finite sets}, the surjection $\pi_I$ determines an open embedding $j(\pi_I): U(\pi_I) \emb X^I$; in the case that the equivalence relation is implicit, we may denote this open embedding by $j(I): U(I) \emb X^I$. Notice that the morphisms $\phi$ from $I$ to $J$ in $\FinNEq$ are exactly those functions $\alpha: I \to J$ such that the induced map $\phi(\alpha) : X^J \to X^I$ restricts to a map $U(J) \to U(I)$. We will denote this map by $\phi_U(\alpha)$. 
\end{notation}

Borcherds begins by considering the functor categories $\Fun(\Fin^*, A)$ and $\Fun((\Fin^*)^\op, A)$, where $* \in \{\emptyset, \not\equiv\}$. These categories each have a natural symmetric monoidal structure induced by the tensor product on $A$: for two functors $V_1, V_2$ and for $I \in \Fin^*$, we have
\begin{align*}
(V_1 \otimes V_2)(I) = V_1(I) \otimes_\BC V_2(I). 
\end{align*}

The next step in the construction involves choosing a coalgebra object in $\Fun((\Fin^*)^\op, A)$. Although the definitions make sense for general coalgebra objects, all our theorems and examples use the same coalgebra, which we will define now. 

\begin{defn}\label{Definition: coalgebra T}
Let $H=\BC[\del]$ be the polynomial algebra in one generator $\del$. It has a natural cocommutative coalgebra structure given by 
\begin{align*}
\Delta(\del) = \del \otimes 1 + 1 \otimes \del.
\end{align*}

Now define the functor $T$ by setting $T(I) = \otimes_{I} H \cong \BC[\del_i]_{i \in I}$. Given a morphism $\alpha: I \to J$, we use the comultiplication on $H$ to define a map
\begin{align*}
\alpha^*&: T(J) \to T(I)\\
 & \del_j \mapsto \Delta^{\alpha^{-1} (j)} (\del).
\end{align*}
(Here $\Delta^{\alpha^{-1} (j)}$ is the $(\vert \alpha^{-1} (j) \vert -1)$-fold composition of $\Delta$ with itself, taking values in $\BC[\del_i]_{i \in \alpha^{-1} (j)}$: that is, $\Delta^{\alpha^{-1} (j)} (\del) = \sum_{i \in \alpha^{-1}(j)} \del_i$.) It is not hard to show that this makes $T$ into a contravariant functor on $\Fin^*$. 

Finally, we define a comultiplication map $\Delta^T : T \to T \otimes T$ by using the tensor products of $\Delta: H \to H \otimes H$ with itself. Coassociativity of $\Delta$ ensures that $\Delta^T$ is indeed a natural transformation, and makes $T$ into a cocommutative coalgebra. 
\end{defn}

From now on, we will work only with this coalgebra object $T$. Borcherds next defines what it means for this coalgebra to \emph{act} on a covariant functor $V$.

\begin{defn}\label{Definition: T action}
Take $V \in \Fun(\Fin^*, A)$. An \emph{action} of $T$ on $V$ consists of a family of maps
\begin{align*}
\act_T(I):& T(I) \otimes V(I) \to V(I)\\
& t\otimes v \mapsto t.v
\end{align*}
satisfying the following compatibility condition: given any $\alpha: I \to J \in \Fin^*$, we have a commutative diagram
\begin{center}
\begin{tikzpicture}[>=angle 90] 
\matrix(b)[matrix of math nodes, row sep=2em, column sep=2em, text height=1.5ex, text depth=0.25ex]
{& T(J) \otimes V(I) & \\
 T(I) \otimes V(I) & & T(J) \otimes V(J) \\
 V(I) & & V(J). \\};
\path[->, font=\scriptsize]
 (b-1-2) edge node[above left]{$\alpha^* \otimes \id$} (b-2-1)
 (b-2-1) edge node[left]{$\act_T(I)$} (b-3-1)
 (b-3-1) edge node[below]{$\alpha_*^V$} (b-3-3)
 (b-1-2) edge node[above right]{$\id \otimes \alpha^V_*$} (b-2-3)
 (b-2-3) edge node[right]{$\act_T(J)$} (b-3-3) ;
\end{tikzpicture}
\end{center}
\end{defn}

We denote the category of functors equipped with $T$-action by $\Fun(\Fin^{*}, A, T)$. (Still, we can work in either $\Fin$, $\FinNEq$. We will begin to see the differences between these categories in the next steps.) We observe, following Borcherds, that the tensor product on $\Fun(\Fin^*, A)$ extends naturally to a tensor product $\otimes$ on $\Fun(\Fin^*, A, T)$. This uses the fact that $T$ is a coalgebra. We now fix a commutative algebra object (with unit) $S$ in the tensor category $\Fun(\FinNEq, A, T)$. The condition that $S$ has a unit is equivalent to requiring that $S(\emptyset)$ is a unital $\BC$-algebra.

\begin{eg}
This is the main example of an algebra object used by Borcherds in \cite{BorQVA}; for this reason we will denote it by $S_B$. For $I$ a finite set with an equivalence relation $\sim$, let 
\begin{align*}
S_B(I) = \BC[(x_i - x_j)^{\pm 1}]_{ i \not\sim j \in I}.
\end{align*}
Given $\alpha : I \to J \in \FinNEq$, we define
\begin{align*}
\alpha^{S_B}_* : S_B(I) \to S_B(J)
\end{align*}
by sending $x_i$ to $x_{\alpha(i)}$. Note that it is crucial that $\alpha$ be a morphism in $\FinNEq$ (i.e. be inequivalence-preserving) in order for this to give a well-defined map. Thus $S_B$ is an object of $\Fun(\FinNEq, A)$. 

The action of $T$ on $S_B$ is given by letting $\del_i \in T(I)= \BC[\del_i]_{i \in I}$ act on $S_B(I)$ by $\frac{d}{dx_i}$. The diagram in definition \ref{Definition: T action} commutes as an immediate consequence of the product rule, so $S_B$ is an object of $\Fun(\FinNEq, A, T)$.

Finally, to show that $S_B$ is a commutative algebra object in this category, we need to give a multiplication morphism $m: S_B \otimes S_B \to S_B$. On the level of a set $I$, $m_I: S_B(I) \otimes S_B(I) \to S_B(I)$ is just the ordinary algebra multiplication; it is easy to check from the definitions (and again using the product rule) that this is defines a natural transformation compatible with the action of $T$. Furthermore, it follows from the fact that the $S_B(I)$ themselves are commutative algebras that the map $m$ makes $S_B$ into a commutative algebra object. 
\end{eg} 

\begin{rmk}\label{Remark: restriction of Sb is trivial}
Note that we can consider $S_B$ as a commutative algebra object in $\Fun(\Fin, A, T)$, by restricting along the embedding $\Fin \emb \FinNEq$; however, in this case the functor is constant, always producing $\BC$, and the action of $T$ is always trivial. 
\end{rmk}

Before moving on, let us also introduce some more examples of algebra objects $S$, not studied by Borcherds. 

\begin{eg} \label{Example: S on A1}
For $I \in \FinNEq$, we set
\begin{align*}
\Sa(I) = \BC[x_i]_{i \in I} [ (x_i - x_j)^{-1}]_{i \not\sim j \in I}.
\end{align*}
The construction of the structure of $\Sa$ as a commutative algebra object in $\Fun(\FinNEq, A, T)$ is analogous to the case of $\Sb$. More precisely, the functions $\alpha_*^{\Sa}$ are induced by $x_i \mapsto x_{\alpha(i)}$, elements $\del_i \in T(I)$ act on $\Sa(I)$ by $\frac{d}{dx_i}$, and the algebra structure of each $\Sa(I)$ is compatible with these structures. 
\end{eg}

\begin{geom}
Let $X = \BBA^1$. Recall from \ref{Notation: finite sets with equivalence relation} that each $I \in \FinNEq$ gives an open subscheme $U(I)$ of $X^I$; note that $\Sa(I) = \Gamma(U(I), \CO_{X^I})$, and the action of $T(I)$ on $\Sa(I)$ is just that of the vector fields $\frac{d}{dx_i}$ on this ring of sections. Motivated by this, we generalize Example \ref{Example: S on A1} as follows.
\end{geom}

\begin{eg} \label{Example: S_X}
Let $X$ be an open affine subset of $\BBA^1$, with fixed global coordinate $x$, as in Section \ref{Subsec: notation}. 

For any $I \in \FinNEq$, the cartesian product $X^I$ is affine, and (because we are working over curves) the open set $U(I)$ is also affine. Let 
\begin{align*}
S_X(I) = \Gamma(U(I), \CO_{X^I}).
\end{align*}
For $\alpha: I \to J$, recall that we have $\phi_U(\alpha): U(J) \to U(I)$. The required map $\alpha^{S_X}_*: S_X(I) \to S_X(J)$ is just the induced map on global sections. 

The action of $T(I)$ on $S_X(I)$ is again by letting $\del_i$ act by the vector field $\frac{d}{dx_i}$ (using the fixed global coordinate $x_i$ on the $i$th factor of $X^I$), and the algebra structure is again the natural one on each ring of functions $S_X(I)$. The compatibility conditions follow as in the previous examples. 

Note that this notation is consistent with the notation introduced in Remark \ref{Remark: affine chiral algebras}.
\end{eg}

With these examples in mind as prototypes, let us move on with Borcherds's construction. He defines the category $\Fun(\FinNEq, A, T, S)$ of $S$-modules: objects are objects $V$ of $\Fun(\FinNEq, A, T)$ equipped with an action map 
\begin{align*}
\act_S: S \otimes V \to V,
\end{align*}
which is a morphism in  $\Fun(\FinNEq, A, T)$, and which satisfies the usual axioms of a ring action. 

\begin{defn}\label{Definition: ordinary tensor product over S}
We define a symmetric tensor product on $\Fun(\FinNEq, A, T, S)$ as follows: for $V, W \in \Fun(\FinNEq, A, T, S)$, the tensor product $V \otimes W$ is the functor given by 
\begin{align*}
(V \otimes W): I \mapsto V(I) \otimes_{S(I)} W(I).
\end{align*}
It is straightforward to give the rest of the structure of $V \otimes W$ as an object of $\Fun(\FinNEq, A, T, S)$, and to see that this makes $\Fun(\FinNEq, A, T, S)$ into a symmetric monoidal tensor category. 
\end{defn}

We will also define a second tensor product, which Borcherds refers to as the \emph{singular tensor product}. 

\begin{defn}\label{Definition: singular tensor product}
Let $V, W \in \Fun(\FinNEq, A, T, S)$. We define $V \odot W$ to be the unique object of $\Fun(\FinNEq, A, T, S)$ such that for any other $Z \in \Fun(\FinNEq, A, T, S)$ a morphism $\psi: V \odot W \to Z$ consists of
\begin{itemize}
\item for any $I_1, I_2$ in $\FinNEq$, a morphism $\psi_{I_1, I_2}: V(I_1) \to W(I_2) \to Z(I_1 \sqcup I_2)$ (here $I_1 \sqcup I_2$ has the equivalence relation given by the disjoint union of the equivalence relations on $I_1$ and $I_2$). 
\end{itemize}
This collection of maps is required to satisfy some natural compatibility conditions:
\begin{itemize}
\item For fixed $I_1$, $I_2$, $\psi_{I_1, I_2}$ should commute with the actions of $T(I_1)$ and $T(I_2)$, and with the actions of $S(I_1)$ and $S(I_2)$.
\item The $\psi$ should be functorial in $I_1$ and $I_2$. 
\end{itemize}
\end{defn}

\begin{rmk}\label{Remark: singular tensor product}
\begin{enumerate}
\item See the remarks following definition 3.10 of \cite{BorQVA} for a discussion of why this object is defined (i.e. representable). 
\item Borcherds also remarks that if we were to attempt to make this definition in the category $\Fun (\Fin, A, T, S)$, we would recover the ordinary tensor product. This amounts to the fact that disjoint union is a coproduct in $\Fin$, but not in $\FinNEq$. 
\item On the other hand, given an object $V$ of $\Fun (\Fin, A, T, S)$, we can extend it to an object $V \in \Fun(\FinNEq, A, T, S)$ by setting
\begin{align*}
V(I_1: \ldots : I_n) \defeq V(I_1, \ldots, I_n) \otimes_{S(I_1) \otimes \cdots \otimes S(I_n)} S(I_1:\ldots : I_n).
\end{align*}
(We adopt the notation $I_1: \ldots : I_n$ to indicate the set $I_1 \sqcup \ldots \sqcup I_n$ with equivalence classes $I_1, \ldots, I_n$, and $I_1, \ldots, I_n$ to indicate the set $I_1 \sqcup \ldots \sqcup I_n$ with all elements in the same equivalence class.) Here $S$ acts in the obvious way by multiplication on the right factor, while $\del_i \in T(I)$ acts by $\del_i \otimes \id + \id \otimes \frac{\del}{\del x_i}$. This gives a fully faithful embedding of $\Fun(\Fin, A, T, S)$ into $\Fun(\FinNEq, A, T, S)$. 

Borcherds observes that if we start with $V, W \in \Fun (\Fin, A, T, S)$, view them as objects of $\Fun(\FinNEq, A, T, S)$, and compute the singular tensor product in this category, the resulting object $V \odot W$ is no longer an object of $\Fun (\Fin, A, T, S)$. 
\end{enumerate}
\end{rmk}

Finally, we can come to the main definition of \cite{BorQVA}:
\begin{defn}[Definition 3.12, \cite{BorQVA}] \label{Definition: (A,H,S)-vertex algebra}
An \emph{$(A, H, S)$-vertex algebra} is a singular commutative ring object in $\Fun (\Fin, A, T, S)$.
\end{defn}

\begin{rmk} [Remark on notation]
Borcherds's notation reflects that he allows for more general additive symmetric monoidal categories $A$ than we use in this paper (we take always $A = \Vect$); likewise, his $H$ can be more general than our $H = \BC[\del]$ used in defining $T$ in definition \ref{Definition: coalgebra T}, and hence he considers slightly more general coalgebras than our $T$. We will still use this notation for consistency, even though for us the only choice is in the algebra object $S$. 
\end{rmk}

\begin{rmk} [Remark on unit condition] \label{Remark: units for AHS}
Fix a triple $(A,H,S)$. An $(A, H, S)$-vertex algebra $V$ is called \emph{unital} if $V(\varnothing)$ is a unital $S(\varnothing)$-algebra. Although Borcherds does not explicitly mention it, it seems likely that he intended all of his $(A,H,S)$-vertex algebras to be unital, and simply left out this axiom. (For example, this assumption is required for his theorems to be true.) It is a consequence of his axioms that $V(\varnothing)$ is an algebra over $S(\varnothing)$, but it is not necessary that it has a unit. On the other hand, a truly \emph{non-unital} $(A,H,S)$-vertex algebra, in the sense of non-unital factorization algebras, or vertex algebras without vacuum, should be a functor $V$ only defined over the category of non-empty finite sets and surjections, together with the rest of the data of an $(A, H,S)$-vertex algebra. 

Let us also remark that an $(A,H,S)$-vertex algebra is unital if and only if it is equipped with a map of $(A,H,S)$-vertex algebras
\begin{align*}
S \to V
\end{align*}
which satisfies the following compatibility condition relating singular-multiplication by the unit to the action of $S$ on $V$: take $I, J$ any finite sets; then the diagram below must commute. 
\end{rmk}

\begin{equation}\label{diagram: unit condition}
\begin{tikzpicture}[>=angle 90] 
\matrix(b)[matrix of math nodes, row sep=2em, column sep=2em, text height=1.5ex, text depth=0.25ex]
{S(I) \otimes V(J) & V(I) \otimes V(J) & V(I:J) \\
S(I, J) \otimes V(I, J) & & V(I, J)   \\};
\path[->, font=\scriptsize]
 (b-1-1) edge node[above]{unit} (b-1-2)
 (b-1-2) edge node[above]{$\mu$} (b-1-3)
 (b-1-1) edge (b-2-1)
 (b-2-1) edge node[below]{act}(b-2-3);
\path[right hook->, font=\scriptsize]
 (b-2-3) edge (b-1-3);
\end{tikzpicture}
\end{equation}
(Here again we require that $S(\varnothing)$ is unital, and $1 \in S(\varnothing)$ acts trivially on $V(\varnothing)$.) 

From now on, we will assume that all $(A, H, S)$-vertex algebras are unital. 

\begin{rmk} [Remark on morphisms] \label{Remark: morphisms}
Borcherds does not explicitly define morphisms between $(A, H, S)$-vertex algebras. For now, we will define the category $\VA(A, H, S)$ to have objects as in definition \ref{Definition: (A,H,S)-vertex algebra}, and morphisms natural transformations of the underlying functors which respect all of the additional structure. See Remark \ref{Remark: modify morphisms} for a discussion of why we may wish to modify the definition of a morphism. 
\end{rmk}

\begin{geom} \label{Remark: geometric interpretation of V}
In the case that $S = S_X$ as in examples \ref{Example: S on A1} and \ref{Example: S_X} above, we can interpret the data of an $(A, H, S)$-vertex algebra geometrically. Indeed, let $V$ be an $(A, H, S_X)$-vertex algebra. Recall that $X^I$ is affine for any finite set $I$, and let $\CV(I)$ denote the $\CO_{X^I}$-module with global sections given by $V(I)$. The generators of $T(I) = \BC[\del_i]_{i \in I}$ act on $V(I)$ by derivations; together with the action of $\Sgen{X}(I)$, this generates a left $\CD_{X^I}$-module structure on $\CV(I)$.

Given any $\alpha: I \to J$, we have $\alpha_*^V: V(I) \to V(J)$. It is straightforward to check from the axioms that this descends to a map 
\begin{align}
S_X(J) \otimes_{S_X(I)} V(I) \to V(J), \label{eq: pullback}
\end{align}
which is compatible with both the $S_X(J)$-action on both sides, and the ``action'' of $T(J)$ (given on the right hand side from the definitions, and given on the left hand side by letting $\del_j$ act by $\del_j \otimes \id + \id \otimes \alpha^*(\del_j)$). To translate this into the $\CD$-module picture, we consider the morphism $\phi(\alpha) : X^J \to X^I$. Then the map (\ref{eq: pullback}) is the morphism on global sections induced by a map of left $\CD_{X^J}$-modules
\begin{align*}
\widetilde{\alpha}: \phi(\alpha)^*(\CV(I)) \to \CV(J).
\end{align*}
These maps are compatible with compositions $I \to J \to K$ in the obvious way. 

For $I \in \FinNEq$, let $\overline{I}$ denote the underlying finite set (i.e. the result of forgetting the equivalence relation). The data of the $V(I)$ for $(\pi_I : I \surj I^\prime) \in \FinNEq$ simply encodes the restriction of $\CV(\overline{I})$ to $U(I) \subset X^I$. Likewise, a function $\alpha: I \to J$ of finite sets is a morphism in $\FinNEq$ if and only if the induced morphism $\phi(\alpha): X^J \to X^I$ sends $U(J)$ to $U(I)$. Then the data of $\alpha_*: V(I) \to V(J)$ encodes the restriction of $\widetilde{\alpha}$ to $U(J)$. 
\end{geom}

Borcherds's main result is the following:
\begin{thm}\cite{BorQVA}\label{Theorem: Borcherds}
Given an $(A, H, S_B)$-vertex algebra $V$, the vector space $V(1)$ has a natural structure of vertex algebra. 
\end{thm}

In fact, it is not hard to see from the proof that this extends to a functor $\VA(A, H, S_B) \to \VA$, which we will denote by $\Phi$. Borcherds gives an instance of this theorem by constructing an $(A, H, S_B)$-vertex algebra $\BV^L$ for which $\Phi(\BV^L)$ is the lattice vertex algebra $V_L$. 

\begin{eg}\label{Example: lattice AHS}
Let $L$ be an even integral lattice of rank $n$, with bilinear form $(\cdot, \cdot)$ as in example \ref{Example: lattice VOA}. Let $V^L$ be the commutative algebra
\begin{align*}
V^L \defeq \BC[L] \otimes \Sym(L(1) \oplus L(2) \oplus \cdots \oplus L(n) \oplus \cdots ).
\end{align*}
(Here we have the symmetric algebra on the direct sum of countably many copies of $L$, denoted by $L(1), L(2)$, etc.) For future reference, let us fix a basis $\{\lambda_a\}_{a \in A}$ of $L$, and let us denote the generators of this algebra by $e^\lambda \in \BC[L]$ (for $\lambda \in L$), and by $\lambda_a(k) \in L(k)$ (for $a \in A$ and $k \in \BN$). We define an action of $\BC[\del]$ on $V^L$ by setting 
\begin{align*}
\del(e^{\lambda_a}) &= \lambda_a(1);\\
\del(\lambda_a(k)) &= (k+1)\lambda_a(k+1),
\end{align*}
and requiring that $\del$ satisfies the Leibniz rule. In particular, denoting $\frac{1}{k!} \del^k$ by $\del^{(k)}$, we have $\lambda_a(k) = \del^{(k)} e^{\lambda_a}$. In addition to the obvious algebra structure, $V^L$ has a coalgebra structure compatible with the algebra structure and the action of $\del$ determined by requiring that 
\begin{align*}
\Delta(e^\lambda) &= e^\lambda \otimes e^\lambda \text{ for any } \lambda \in L;\\
\Delta(v_1 v_2) &= \Delta(v_1)\cdot \Delta(v_2) \text{ for any } v_1, v_2 \in V^L;\\
\Delta(\del v) &= (\del \otimes 1 + 1 \otimes \del) \Delta(v) \text{ for any } v \in V^L. 
\end{align*}

Define $\BV^L : \Fin \to \CA$ by $\BV^L(I) \defeq \otimes_{I} V^L$. Given a morphism $\alpha: I \to J$, we define $\BV^L(I) \to \BV^L(J)$ by $(v_i)_{i \in I} \mapsto (\prod_{i \in \alpha^{-1}(j)} v_i)_{j \in J}$. Then we can define the action of $T(I)=\BC[\del_i]_{i \in I}$ on $\BV^L(I)$ by letting $\del_i$ act on the $i$th factor $V^L$. 

The action of $S_B$ is given by the natural action of $S_B(I_1: \ldots : I_n)$ on $\BV^L (I_1: \ldots : I_n) = \BV^L(I_1, \ldots, I_n) \otimes_\BC S_B(I_1: \ldots: I_n)$. 

It remains to define the singular multiplication $\mu$ on $\BV^L$. This should be a map $\BV^L \odot \BV^L \to \BV^L$, or in other words a collection of compatible maps
\begin{align*}
\mu_{I_1, I_2}: \BV^L(I_1) \otimes \BV^L(I_2) \to \BV^L(I_1, I_2) \otimes_{\BC} \BC[(x_{i_1} - x_{i_2})^{\pm 1} ]_{i_1 \in I_1, i_2 \in I_2}. 
\end{align*}
We define these maps in several steps. Fix a two-cocycle $c: L \times L \to \{\pm1\}$ as in example \ref{Example: lattice VOA}. Now define what Borcherds calls a \emph{bicharacter} $r: V^L \otimes V^L \to \BC[(x_1 - x_2)^{\pm 1}]$ as follows. For $\alpha, \beta \in L$, set $r(e^\alpha, e^\beta) = c(\alpha, \beta) (x_1 - x_2)^{(\alpha, \beta)}$. Note that since $c$ satifies $c(\alpha, \beta) = (-1)^{(\alpha, \beta)}c(\beta, \alpha)$, we have $r(e^\beta, e^\alpha) = c(\alpha, \beta) (x_2 - x_1)^{(\alpha, \beta)}$. The value of $r$ on any element of $V^L \otimes V^L$ is completely determined by the requirement that $r$ is bilinear and satisfies\footnote{More generally, these are the axioms given by Borcherds for a \emph{(symmetric) bicharacter} from an arbitrary commutative cocommutative bialgebra into $\BC[(x_1 - x_2)^{\pm 1}]$; he uses this general set-up to produce other examples of $(A, H, S_B)$-vertex algebras, but this example is sufficient for our purposes.}
\begin{itemize}
\item $\text{if } r(v \otimes v) = f(x_1, x_2), \text{ then } r(w \otimes v) = f(x_2, x_1) \text{ for } v, w, \in V^L;$
\item $r(\del v \otimes w) = \frac{\del}{\del x_1} r(v \otimes w) \text{ for } v, w \in V^L;$
\item $r(v_1v_2 \otimes w) = r(v_1, w_{(1)})r(v_2, w_{(2)}), \text{ for } v_1, v_2, w \in V^L, \text{ where } \Delta(w) = w_{(1)} \otimes w_{(2)}.$
\end{itemize}
Borcherds explains (Lemma 4.1, \cite{BorQVA}) that this bicharacter can be extended to give maps
\begin{align*}
r:  \BV^L(I_1) \otimes \BV^L(I_2) \to S_B(I_1: I_2),
\end{align*}
for any pair $I_1, I_2$, using the coalgebra structure of $V^L$ and the algebra structure of $S_B(I_1: I_2)$. With this in hand, we define $\mu_{I_1, I_2}$ as follows: given $v \in \BV^L(I_1)$ and $w \in \BV^L(I_2)$, write $\Delta(v) = v_{(1)} \otimes v_{(2)} \in \BV^L(I_1) \otimes \BV^L(I_1)$, and likewise write $\Delta(w) = w_{(1)} \otimes w_{(2)} \in \BV^L(I_2) \otimes \BV^L(I_2)$. Then we define
\begin{align} \label{Eq: definition of mu}
\mu_{I_1, I_2} (v \otimes w) \defeq (v_{(1)} \otimes w_{(1)}) \otimes r(v_{(2)} \otimes w_{(2)}). 
\end{align}
It is not hard to check that this gives a compatible family of maps, which in turn gives a commutative multiplication map $\BV^L \odot \BV^L \to \BV^L$. 

Although we will not explain here why the vertex algebra structure on $\Phi(\BV^L)$ matches that of the lattice vertex algebra $V_L$ (see Example \ref{Example: lattice VOA}), let us at least explain how the underlying vector spaces $\BV^L(1) = V^L$ and $V_L$ are identified, by defining the isomorphism
\begin{align}\label{Eq: identifying V^L and V_L}
f: V^L \to V_L.
\end{align}
In fact, both $V^L$ and $V_L$ are commutative algebras with derivation (given by the action of $\partial$ on $V^L$ and the action of $T$ on $V_L$, respectively). Moreover, each is generated under multiplication and the derivation by elements parametrized by $L$: $\{e^\lambda\}_{\lambda \in L}$ in the case of $V^L$ and $\{ | \lambda \rangle\}_{\lambda \in L}$ in the case of $V_L$. The map $f$ is the homomorphism of commutative algebras with derivation given by $f(e^\lambda) = | \lambda \rangle$. 

As an example, we can calculate that for a basis element $\lambda_a$, $f(\lambda_a(1)) = f(\partial e^{\lambda_a} ) = T | \lambda_a \rangle = (\lambda_a)_{-1} | \lambda_a \rangle$. 
\end{eg}

\begin{rmk}[Remark on bicharacter constructions] \label{Remark: bicharacter}
In \cite{Patnaik}, Patnaik proves carefully that $\Phi(\BV^L)$ is indeed the lattice vertex algebra, as Borcherds's claims. He also defines and characterizes the class of \emph{$r$-vertex algebras}, vertex algebras which arise as objects in the image of $\Phi$ of $(A,H,S)$-vertex algebras which are constructed in a way similar to the above method, namely beginning with a commutative cocommutative bialgebra and a symmetric bicharacter. In particular, not all vertex algebras arise in this way; this result is different from the results of the current paper, which concern the application of $\Phi$ (or $\Phi_X$) to \emph{any} $(A, H, S)$-vertex algebra. 
\end{rmk}

\begin{rmk}[Remark on commutativity condition] \label{Remark: commutativity}
Let us describe more explicitly what it means for a singular multiplicative structure $\mu$ on an object $V \in \Fun(\Fin, \CA, T, S)$ to be \emph{commutative}. In terms of the components $\mu_{I_1, I_2}$, the condition simply amounts to the commutativity of the following diagrams:
\begin{center}
\begin{tikzpicture}[>=angle 90] 
\matrix(b)[matrix of math nodes, row sep=2em, column sep=2em, text height=1.5ex, text depth=0.25ex]
{ V(I_1) \otimes V(I_2) & V(I_1: I_2) \\
 V(I_2) \otimes V(I_1) & V(I_2: I_1)\\};
\path[->, font=\scriptsize]
(b-1-1) edge node[above]{$\mu_{I_1, I_2}$} (b-1-2)
(b-1-1) edge node[left]{$\alpha \otimes \beta \mapsto \beta \otimes \alpha$} (b-2-1)
(b-1-2) edge node[right]{$\id$}(b-2-2)
(b-2-1) edge node[below]{$\mu_{I_1, I_2}$} (b-2-2); 
\end{tikzpicture}
\end{center}
However, when we are dealing with explicit examples, say $I_1$ and $I_2$ each consisting of one point, there is a risk of notational confusion. Rather than labelling the element of $I_1$ by ``$1$'' and the element of $I_2$ by ``$2$'', we do not specify the labels, and reserve the notation $x_1$ to refer to the first label appearing on the map $\mu$---i.e. to $I_1$ in the first row of the diagram, and to $I_2$ in the second row---while $x_2$ refers to the second label. We do this for consistency with notation of \cite{BorQVA} and for ease of translation to the (ordinary) vertex algebra setting. Thus in this case, to express the commutativity of the above diagram, we use the transposition $\sigma: \{1,2\} \to \{2,1\}$, and we see that 
\begin{align}\label{Eq: commutativity condition}
\mu_{12}(\alpha \otimes \beta) = \sigma^V_* \mu_{1,2}(\beta \otimes \alpha). 
\end{align}
The interested reader can check that this formula holds in the example of the lattice $(A, H, S)$-vertex algebra $\BV^L$. 
\end{rmk}

\begin{rmk}[Generalization to families]\label{rmk: AHS in families}
Borcherds's definitions work for a general commutative ring $R$, not just over $\BC$. In particular, consider the case when $R = \BC[t]$ and $X$ is an affine curve in $\BBA^1_R$, viewed as a scheme over $\Spec R = \BBA^1$; the case we will use later on is the special setting of $X = \BBA^1 \times \BBA^1$, mapping to $\Spec R$ via the first projection.  The theory is entirely analogous to the complex setting; let us mention only a few differences. 
\begin{itemize}
\item We consider functors from $\Fin^*$ to the category of $\BC[t]$ modules (our new category $A$). 
\item The coalgebra object $T$ is as before (with $\BC[t]$-module structure given by $t\mapsto 0$). 
\item The algebra object $S_X$ comes from considering functions on the iterated fibre products $X^I_{\BBA^1} \defeq X \times _{\BBA^1} X \times _{\BBA^1} \cdots \times_{\BBA^1} X$ and their open subschemes $U_{\BBA^1}(I)$ defined analogously to the non-relative setting. (In the case $X = \BBA^1 \times \BBA^1$ mentioned above, $X^I_{\BBA^1} \cong X \times X^I $ and $U_{\BBA^1}(I) \cong X \times U(I)$.)
\item The structure given by the action of $T$ and of $S$ makes $V(I)$ into a relative $\CD$-module on $X^I_{\BBA^1}$ over $\BBA^1$. 
\item The naive tensor product is defined as before, by tensoring over $S_X(I)$. 
\item The singular tensor product is defined by considering compatible families of maps on tensor products of modules over $\BC[t]$ (rather than over $\BC$). 
\end{itemize}
We will call such an $(A,H,S)$-vertex algebra a \emph{family of $(A,H,S)$-vertex algebras over $\BBA^1$}. 
\end{rmk}

\section{From \texorpdfstring{$(A,H,S_X)$}{(A,H,S)}-vertex algebras to vertex algebras on \texorpdfstring{$X$}{X}}\label{Section: from AHS to vertex algebras}
In this section, we generalize Borcherds's result (Theorem \ref{Theorem: Borcherds}) to deal with our other class of algebra objects $S$: namely, we return to the setting of $X \subset \BBA^1$ open affine, with a fixed global coordinate $x$. 

\begin{thm}\label{Theorem: AHS to vertex algebra}
For $X$ as above, and $S_X$ the algebra object defined in example \ref{Example: S_X}, we have a functor
\begin{align*}
\Phi_X : \VA(A, H, S_X) \to \VA(X). 
\end{align*}
\end{thm}

\begin{proof}
For $\BV \in \VA(A, H, S_X)$, we will define a structure of vertex algebra over $X$ on the vector space $V = V(1)$. (For convenience, we will identify $V$ with $V(I)$ for any singleton set $I$, using the isomorphism induced by the unique isomorphism $I \to \{1\}$.)

The key idea in the construction is the generalization of Borcherds's ``Taylor series expansion function'' to the $\Sgen{X}$ setting. More precisely, we consider the map $\alpha_*^V: V(1,2) \to V(1)$ corresponding to $\alpha: \{1,2\} \to \{1\}$, and define
\begin{align*}
\theta: V(1,2) & \to V(1) \otimes_\BC \BC [\![ z_1,z_2 ]\!]\\
v & \mapsto \sum_{i, j \geq 1} \alpha_*^V(\del_1^{(i)} \del_2^{(j)} v) \otimes z_1^i z_2^j.
\end{align*}
Observe that $\theta$ is compatible with the action of $x_1, x_2 \in S_X(1,2)$ on $V(1,2)$ in the following way: 
\begin{align}
\theta (x_1.v) &= (x \otimes 1 + 1 \otimes z_1).\theta (v), \label{eq: theta} \\
\theta (x_2.v) &=(x \otimes 1 + 1 \otimes z_2).\theta (v). \nonumber
\end{align}
In particular, $\theta$ is linear over $\BC[(x_1 - x_2)]$, where $x_1 - x_2$ acts on the right by $z_1-z_2$. It follows that $\theta$ extends to a map on the localization at $(x_1 - x_2)$: 
\begin{align*}
\theta: V(1,2) \otimes_{S_X(1,2)} S_X(1:2) \to V(1)\otimes_{\BC}\BC[\![z_1,z_2]\!][(z_1-z_2)^{-1}].
\end{align*}

From the singular multiplication $\mu$ on $\BV$, we have in particular a map
\begin{align*}
\mu_{12}: V(1) \otimes V(2) \to V(1,2) \otimes_{S_X(1,2)} S_X(1:2).
\end{align*}
For $v, u \in V$ we define $Y(v, z_1)u \in V(1) \otimes_{\BC}\BC[\![z_1]\!][z_1^{-1}]$ by 
\begin{align*}
Y(v,z_1)u \defeq \theta(\mu(v\otimes u))_{\vert z_2=0}.
\end{align*}

More explicitly, choose $N$ large enough that $(x_1 - x_2)^N \mu_{12}(v \otimes u)\defeq a \in V(1,2)$. Then 
\begin{align}\label{Eq: Y}
Y(v,z_1)u = \sum_{i\ge 0} \alpha^V_*(\del_1^{(i)} a) \otimes z_1^{i-N}.
\end{align}

We claim that $Y(\cdot, z)(\cdot)$ makes $V$ into a vertex algebra on $X$, together with the data of vacuum vector $\vac$ given by the image of $1 \in S_X(1)$ under the unit morphism $S_X(1) \to V(1)$, and of translation operator given by $T = \del(\cdot): V \to V$. Let us check the axioms.

First notice that the lower truncation property is immediate from the construction of $Y$. It is also immediate from the properties of an $(A, H, S_X)$-vertex algebra that $V$ is a $\Gamma(X, \CO_X)$-module. It easy to check from the linearity of $\mu$ and the properties (\ref{eq: theta}) of $\theta$ that we have
\begin{align*}
Y(xv,z)u & = (x+z) Y(v,z)u;\\
Y(v,z)(xu) & = xY(v,z)u,
\end{align*}
which are the $\CO(X)$-linearity properties required. The derivative property is also straightforward to check. 

Next let us check the vacuum axiom and the creation axiom. Since $\del(1) = 0$, we also have $\del(\vac) = 0$. For any other $v \in V$, the diagram (\ref{diagram: unit condition}) tells us that $\mu(\vac \otimes v) = \gamma^V_*v \in V(1,2)$, where $\gamma$ is the inclusion $\{2\} \hookrightarrow \{1,2\}$. Then from (\ref{Eq: Y}) we see that
\begin{align*}
Y(\vac, z) v &= \sum_{i \ge 0} \alpha^V_*(\del_1^{(i)} (\mu (\vac \otimes v)) \otimes z^i \\
 & = \sum _{i \ge 0} \alpha^V_*(\mu( T^{(i)}\vac \otimes v)) \otimes z^i\\
 & = \alpha^V_*(\mu(\vac \otimes v))\\
 & = \alpha^V_* \gamma^V_* v = v.
\end{align*}
Thus the vacuum condition holds. 

Likewise, $\mu(\vac \otimes v) = \delta^V_*v \in V(1,2)$, for $\delta: \{1\} \hookrightarrow \{1,2\}$, so 
\begin{align*}
Y(v, z) \vac &= \sum_{i \ge 0} \alpha^V_* (\del_1^{(i)} (\mu (v \otimes \vac) )\otimes z^i\\
 & = \sum_{i \ge 0} \alpha^V_*(\mu (\del^{(i)}v \otimes \vac))  \otimes z^i\\
 & = v + Tv \otimes z + \cdots.
 \end{align*}
 This proves that the creation property is satisfied. 
 
 Finally, we claim that the following proposition holds:
 \begin{prop} \label{Proposition: locality holds}
 The locality axiom is satisfied. In other words, for any three elements $u_1, u_2, u_3 \in V$ there exists some large integer $N \gg 0$ such that
\begin{align*}
(z_1 - z_2)^N Y(u_1, z_1) Y(u_2, z_2) u_3 = (z_1 - z_2)^N Y(u_2, z_2) Y(u_1, z_1) u_3.
\end{align*}
 \end{prop}
 
 We defer the proof of this proposition for now; in the mean time, assuming the proposition, we see that we have indeed constructed the data of a vertex algebra $(V, \vac, T, Y(\cdot, z))$ over $X$. 
 
 It is also easy to see from the constructions that a morphism $F: \BV \to \BW$ of $(A, H, S_X)$-vertex algebras (in the sense of remark \ref{Remark: morphisms}) gives rise to a linear map $F(1) : V(1) \to W(1)$ which respects all of the vertex algebra structure just described. This describes the behaviour of the functor $\Phi_X$ on morphisms, and completes the proof of the theorem.
\end{proof}

The remainder of this section is devoted to the proof of Proposition \ref{Proposition: locality holds}. For $u_1, u_2, u_3 \in V$, we need to show that the two elements $ Y(u_1, z_1) Y(u_2, z_2) u_3$ and  $Y(u_2, z_2) Y(u_1, z_1) u_3$ of $V(1) \otimes \BC[\![ z_1, z_2 ]\!] [z_1^{-1}, z_2 ^{-1}]$ agree when viewed as elements of $V(1) \otimes \BC[\![ z_1, z_2 ]\!] [z_1^{-1}, z_2 ^{-1}, (z_1 - z_2)^{-1}]$. Following Borcherds's proof in the setting of $(A, H, S_B)$-vertex algebras, we will introduce an intermediate function that will allow us to see how the associativity and commutativity of the singular multiplication map $\mu$ imply our result. 

For this, we use a generalization of the function $\theta$ to three variables: let $\beta$ be the function $\{1,2,3\} \to \{1\}$, and define
\begin{align*}
\theta_{123} &: V(1:2:3) \to V(1) \otimes \BC[\![ z_1, z_2, z_3 ]\!] [(z_1 - z_2) ^{-1} , (z_1 - z_3) ^{-1}, (z_2 - z_3) ^{-1}] 
\end{align*}
by sending $(x_1 - x_2)^{d_1}(x_1 - x_3)^{d_2}(x_2-x_3)^{d_3} a$ for $a\in V(1,2,3)$ to 
\begin{align*}
\sum_{i, j, k \ge 0} \beta^V_* (\del^{(i)}_1 \del^{(j)}_2 \del^{(k)}_3 a ) \otimes z_1^iz_2^jz_3^k (z_1 - z_2)^{d_1}(z_1 - z_3)^{d_2}(z_2-z_3)^{d_3}.
\end{align*}

\begin{defn}
Let $U$ be the map 
\begin{align*}
V(1) \otimes V(2) \otimes V(3) \to V(1) \otimes \BC[\![ z_1, z_2 ]\!] [z_1^{-1}, z_2 ^{-1}, (z_1 - z_2)^{-1}]
\end{align*}
given by the following composition:
\begin{align*}
V(1) \otimes V(2) \otimes V(3)  & \xrightarrow{\id \otimes \mu} V(1) \otimes V(2:3) \xrightarrow{\mu} V(1:2:3) \\
 & \xrightarrow{\theta_{123}} V(1) \otimes \BC[\![ z_1, z_2, z_3 ]\!] [(z_1 - z_2) ^{-1} , (z_1 - z_3) ^{-1}, (z_2 - z_3) ^{-1}] \\
 & \xrightarrow{z_3 \mapsto 0} V(1) \otimes \BC[\![ z_1, z_2 ]\!] [z_1^{-1}, z_2 ^{-1}, (z_1 - z_2)^{-1}].
\end{align*}
\end{defn}

\begin{lemma}\label{lemma: U and Y agree}
Let $Y_{123}: V(1) \otimes V(2) \otimes V(3) \to V(1) \otimes \BC[\![ z_1, z_2 ]\!] [z_1^{-1}, z_2 ^{-1}, (z_1 - z_2)^{-1}]$ be the map given by 
\begin{align*}
u_1 \otimes u_2 \otimes u_3 \mapsto Y(u_1, z_1)Y(u_2, z_2)u_3.
\end{align*}
Then $Y_{123} = U$. 
\end{lemma}

\begin{proof}
Consider the following diagram.

\begin{center}
\begin{tikzpicture}[>=angle 90] 
\matrix(b)[matrix of math nodes, row sep=2em, column sep=0em, text height=1.5ex, text depth=0.25ex, font=\tiny]
{V(1) \otimes V(2) \otimes V(3)  & & \\
  V(1) \otimes V(2:3) & {V(1) \otimes V(2) \otimes \BC[\![ z_2, z_3]\!][(z_2 - z_3)^{-1}]} & V(1) \otimes V(2) \otimes \BC[\![ z_2]\!][z_2^{-1}] \\
  V(1:2:3) & V(1:2) \otimes \BC[\![ z_2, z_3]\!][(z_2 - z_3)^{-1}] & V(1:2) \otimes \BC[\![ z_2]\!][z_2^{-1}] \\
  V(1) \otimes \BC[\![ z_1, z_2, z_3]\!][(z_i -z_j)^{-1}]  & & V(1) \otimes \BC[\![ z_1, \tilde{z}_2, z_2]\!][(z_1 - \tilde{z}_2)^{-1}, z_2^{-1}] \\
  & V(1) \otimes \BC[\![ z_1, z_2]\!][(z_1-z_2)^{-1},z_1^{-1}, z_2^{-1}] & V(1) \otimes \BC[\![ z_1, z_2]\!][z_1^{-1}, z_2^{-1}].\\};
  
\path[->, font=\tiny]
(b-1-1) edge node[above right]{$\id \otimes \mu$} (b-2-1)
(b-2-1) edge node[above]{$\id \otimes \theta_{23}$} (b-2-2)
(b-2-2) edge node[above]{$z_3 \mapsto 0$} (b-2-3)
(b-2-1) edge node[left]{$\mu$} (b-3-1)
(b-2-2) edge node[left]{$\mu \otimes \id$} (b-3-2)
(b-2-3) edge node[right]{$\mu \otimes \id$} (b-3-3)
(b-3-1) edge node[below]{$\theta_{23}$} (b-3-2)
(b-3-2) edge node[below]{$z_3 \mapsto 0$} (b-3-3)
(b-3-1) edge node[left]{$\theta_{123}$}(b-4-1)
(b-3-3) edge node[right]{$\theta_{12}$}(b-4-3)
(b-4-3) edge node[right]{$\tilde{z}_2 \mapsto 0$} (b-5-3)
(b-4-1) edge node[below left]{$z_3 \mapsto 0$} (b-5-2)
(b-5-3) edge (b-5-2);  
  
\end{tikzpicture}
\end{center}
Note that the composition around the outer right edge of this diagram is exactly $Y_{123}$, while the composition around the outer left edge is $U$, so in order to prove the lemma it suffices to show that each of the three squares in the diagram commutes. 

The upper left square commutes because of the compatibility of $\mu$ with the functors induced by the surjections $\{2,3\} \to \{2\}$ and $\{1\} \sqcup \{2,3\} \to \{1\} \sqcup \{2\}$ and with the action of the $\del_i$. 

That the upper right square commutes is immediate, so it remains to check that the large bottom square commutes. Let us denote the two compositions by $F_{\urcorner}$ and $F_{\llcorner}$. First note that both functions give maps $V(1:2:3) \to V(1) \otimes \BC[\![ z_1, z_2]\!][(z_1-z_2)^{-1},z_1^{-1}, z_2^{-1}]$ which are linear with respect to the action of $(x_i- x_j)$ for $i \ne j$ in the following sense: for $v \in V(1:2:3)$, $(i,j) \in \{ (1,2), (1,3), (2,3)\}$, and $* = \urcorner, \llcorner$, we have
\begin{align*}
F_* ((x_i - x_j) \cdot v) = \left \{ \begin{array}{l l } (z_1 -z_2) F_*(v) & \text{ if } (i,j) = (1,2); \\ 
 z_1 F_*(v) & \text{ if } (i,j) = (1,3); \\
 z_2 F_*(v) & \text{ if } (i,j) = (2,3). \end{array} \right.
\end{align*}
Since $V(1:2:3)$ is the localization of $V(1,2,3)$ by $(x_1-x_2)(x_1 - x_3)(x_2 - x_3)$ we conclude that it suffices to check that $F_\urcorner(v) = F_\llcorner(v)$ for $v \in V(1,2,3)$. So we take $v \in V(1,2,3)$, and calculate from the definitions that we have
\begin{align*}
F_\llcorner (v) = \sum_{i, j \ge 0} \beta^V_* (\del_1^{(i)} \del_2^{(j)} v) \otimes z_1^{i} z_2^{j} = F_\urcorner(v).
\end{align*}
This completes the proof of the lemma. 
\end{proof}

The locality condition is a straightforward consequence of this lemma, together with the properties of the singular multiplication map $\mu$.

\begin{proof}[Proof of Proposition \ref{Proposition: locality holds}]
First let us fix some notation: we will denote the transposition that swaps $1$ and $2$  by $\sigma$ when viewed as a permutation of $\{1,2\}$, and by $\overline{\sigma}$ when viewed as a permutation of $\{1,2,3\}$. It induces maps 
\begin{align*}
\sigma^V_* &: V(1:2) \to V(1:2);\\
\overline{\sigma}^V_* &: V(1:2:3) \to V(1:2:3);\\
g_\sigma &: \BC[\![ z_1, z_2 ]\!] [z_1^{-1}, z_2^{-1}, (z_1 - z_2)^{-1} ] \to \BC[\![ z_1, z_2 ]\!] [z_1^{-1}, z_2^{-1}, (z_1 - z_2)^{-1} ] ;\\
g_{\overline{\sigma}} &:  \BC[\![ z_1, z_2, z_3]\!][(z_i -z_j)^{-1}] \to \BC[\![ z_1, z_2, z_3]\!][(z_i -z_j)^{-1}].
\end{align*}

Now we make the following observations:
\begin{align*}
Y(u_1, z_1) Y(u_2, z_2) u_3 & = U(u_1 \otimes u_2 \otimes u_3) \\
 & =F_\llcorner (\mu (u_1 \otimes \mu(u_2 \otimes u_3))) \\
 & =F_\llcorner (\mu (\mu (u_1 \otimes u_2) \otimes u_3))\\
 & =F_\llcorner (\mu (\sigma^V_* \mu (u_2 \otimes u_1) \otimes u_3))\\
 & =F_\llcorner (\overline{\sigma}^V_* \mu (\mu (u_2 \otimes u_1) \otimes u_3)).
\end{align*}
Here the first equality is Lemma \ref{lemma: U and Y agree}; the second is by definition of $U$ and $F_\llcorner$; the third is the associativity of $\mu$; the fourth is the commutativity condition (\ref{Eq: commutativity condition}); and the last is the functoriality of $\mu$. 

It is easy to check from the definitions that we have $\theta_{123} \circ \overline{\sigma}^V_*(v) = (\id \otimes g_{\overline{\sigma}}) \circ \theta_{123}(v)$ for any $v \in V(1,2,3)$. Furthermore, the two compositions satisfy the same linearity properties with respect to multiplication by $(x_i -x_j)$, so they agree on the localization $V(1:2:3)$. It follows that 
\begin{align*}
F_\llcorner (\overline{\sigma}^V_* \mu (\mu (u_2 \otimes u_1) \otimes u_3)) = g_\sigma (F_\llcorner (\mu (\mu (u_2 \otimes u_1) \otimes u_3))).
\end{align*}
Arguing as above, we see that this is equal to $g_\sigma (U(u_2 \otimes u_1 \otimes u_3))$, which by Lemma \ref{lemma: U and Y agree} is $g_\sigma (Y(u_2, z_1) Y(u_1, z_2) u_3) = Y(u_2, z_2) Y(u_1,z_1) u_3$.
\end{proof}

\section{From \texorpdfstring{$(A,H, S_X)$}{(A,H,S)}-vertex algebras to chiral algebras on \texorpdfstring{$X$}{X}}\label{Section: from AHS to chiral}
Recall part (1) of Theorem \ref{Theorem: everybody else's theorem} (Thm. 5.4 \cite{HL}). It provides an equivalence of categories
\begin{align*}
\Psi_X: \VA(X) \to \CAlg(X).
\end{align*}
Combining this with Theorem \ref{Theorem: AHS to vertex algebra}, we see that given any $(A, H, S_X)$-vertex algebra $\BV$, we obtain a chiral algebra $\Psi_X \circ \Phi_X (\BV)$ on $X$. In this section, we calculate the chiral bracket $\mu^\ch$ directly from the data of $\BV$, rather than passing through the category of vertex algebras on $X$. 

First let us establish some notation. As in the proof of Theorem \ref{Theorem: AHS to vertex algebra}, let $V = \BV(1)$ be the vector space underlying the vertex algebra $\Phi_X(\BV)$. As in remark \ref{Remark: geometric interpretation of V}, let $\CV = \CV(1)$ denote the left $\CD$-module on $X$ with global sections $V$, and similarly for any finite set $I$ let $\CV(I)$ denote the left $\CD$-module on $X^I$ with global sections $\BV(I)$. From the proof of Theorem 5.4 in \cite{HL} (or perhaps more explicitly seen in similar discussions in Section 19.2 of \cite{FBZ}), the right $\CD$-module underlying the chiral algebra $\Psi_X \circ \Phi_X$ is the right $\CD$-module $\CV^r= \CV \otimes \omega_X$ corresponding to the left $\CD$-module $\CV$. Let us denote this right $\CD$-module by $\CA$.

To define a chiral bracket $\mu_0^\ch : j_* j^* \CA \boxtimes \CA \to \Delta_! \CA$, it is enough to give a map of left $\CD$-modules
\begin{align*}
\lambda_0^\ch: j_* j^* \CV \boxtimes \CV \to \Delta_* \CV.
\end{align*}
We give such a morphism by composing three maps as follows:
\begin{enumerate}
\item The first map $j_* j^* \CV \boxtimes \CV \to j_* j^* \CV(1,2)$ is induced on global sections by the singular multiplication map
\begin{align*}
\mu_{12}: (V(1) \otimes V(2)) \to V(1,2) \otimes_{S_X(1,2)} S_X(1:2),
\end{align*}
extended to a map of the localization
\begin{align*}
(V(1) \otimes V(2))\otimes_{S_X(1,2)} S_X(1:2) \to V(1,2) \otimes_{S_X(1,2)} S_X(1:2).
\end{align*}
\item The second map $\gamma: j_* j^* \CV(1,2) \to \Delta_*\Delta^* \CV(1,2)$ is the last map in the canonical exact sequence
\begin{align*}
0 \to \Delta_* \Delta^! \CV(1,2) \to \CV(1,2) \to j_* j^* \CV(1,2) \to \Delta_* \Delta^* \CV(1,2) \to 0. 
\end{align*}
(Here $\gamma$ is surjective because $j$ is affine and hence there are no higher derived terms for $j_*$.)
\item Notice that for  $\alpha: \{1,2\} \to \{1\}$, the map $\phi(\alpha)$ is just the diagonal embedding $\Delta$. The third map $\Delta_* \Delta^* \CV(1,2) \to \Delta_* \CV$ is induced from the map $\widetilde{\alpha}: \phi(\alpha)^*(\CV(1,2)) \to \CV(1)$ of Remark \ref{Remark: geometric interpretation of V}.
\end{enumerate}

\begin{thm} \label{Theorem: AHS to chiral algebra}
The map $\mu^\ch_0$ defined in this way agrees with the chiral bracket $\mu^\ch$ on the chiral algebra $\Psi_X \circ \Phi_X (\BV) = (\CA, \mu^\ch)$. 
\end{thm}

\begin{proof}
Let us recall the construction of the chiral bracket $\mu^\ch$. (In fact, we adapt the notation of Frenkel--Ben-Zvi \cite{FBZ}, who work only over a formal disk, to the ideas of Huang--Lepowsky \cite{HL}, who work over arbitrary open affine $X \in \BBA^1$.) The bracket $\mu^\ch$ is again defined in terms of a map of left $\CD$-modules\footnote
{In the notation of Frenkel--Ben-Zvi, this map, defined on a formal disk around a point $(x,x) \in X^2$, is called $\CY_x^2$.}
\begin{align}\label{Eq: chiral bracket}
\lambda^\ch: j_* j^* \CV \boxtimes \CV \to \Delta_* \CV.
\end{align}

Note that $j_* j^* \CV \boxtimes \CV$ has global sections $(V \otimes V) \otimes_{S_X(1,2)} S_X(1:2)$, viewed as an $S_X(1,2)$-module. Then $\lambda^\ch$ is given on global sections by 
\begin{align*}
v \otimes u \otimes (x_1 - x_2)^{-n} \mapsto Y(v, x_1 - x_2)u \otimes (x_1 - x_2)^{-n} \text{ mod}\ S_X(1) \otimes V.
\end{align*}
Recalling the definition of $Y(\cdot,z)(\cdot)$ from the proof of Theorem \ref{Theorem: AHS to vertex algebra}, we see that $\lambda^\ch$ is given by
\begin{align*}
v \otimes u \otimes (x_1 - x_2) ^{-n} \mapsto  \theta(\mu(v\otimes u))\vert _{z= (x_1 - x_2); w=0}(x_1 - x_2)^{-n}.
\end{align*}

Next let us write down the map $\gamma$ explicitly in terms of global sections. Global sections of
\begin{align*}
\Delta_* \Delta ^* \CV(1,2) = \frac{ j_* j^* (\CO_X \boxtimes \Delta^* \CV(1,2))}{\CO_X \boxtimes \Delta^* \CV(1,2)}
\end{align*}
are generated over $S_X(1,2)$ by elements that look like $1 \otimes \overline{v} \otimes (x_1-x_2)^{-n} \ \text{mod}\ S_X(1) \otimes V(1,2)/(x - y)V(1,2)$, where $\overline{v} \in V(1,2)/(x - y) V(1,2)$ and $n \in \BZ_{<0}$. Let us quickly recall the $\CD_{X^2}$-module structure in terms of these global sections:
\begin{align*}
x_1: 1 \otimes \overline{v} \otimes (x_1-x_2)^{-n} \mapsto & x \otimes \overline{v} \otimes (x_1-x_2)^{-n} ;\\
x_2: 1 \otimes \overline{v} \otimes (x_1-x_2)^{-n}  \mapsto & 1 \otimes \overline{xv} \otimes (x_1-x_2)^{-n} = 1 \otimes \overline{yv} \otimes (x_1-x_2)^{-n};\\
\del_{x_1}: 1 \otimes \overline{v} \otimes (x_1-x_2)^{-n}  \mapsto & \del_x(1) \otimes \overline{v} \otimes (x_1-x_2)^{-n} + 1 \otimes \overline{v} \otimes \del_{x_1}(x_1-x_2)^{-n}\\ 
& = -n \otimes \overline{v} \otimes (x_1-x_2)^{-n-1} ;\\
\del_{x_2}: 1 \otimes \overline{v} \otimes (x_1-x_2)^{-n}  \mapsto & 1 \otimes \overline{(\del_x + \del_y)v} \otimes (x_1-x_2)^{-n}  + 1 \otimes \overline{v} \otimes \del_{x_2}(x_1-x_2)^{-n} \\
& = 1 \otimes \overline{(\del_x + \del_y)v} \otimes (x_1-x_2)^{-n}  + n \otimes \overline{v} \otimes (x_1-x_2)^{-n-1}.
\end{align*}

The map $\gamma$ is defined by 
\begin{align*}
v \otimes (x_1 - x_2) ^{-n-1} \mapsto \sum_{i = 0}^{n} \frac{1}{i!} \otimes \overline{\del_x^i v} \otimes (x_1 - x_2) ^{i - n -1} \ \text{mod}\ S_X(1) \otimes V(1,2)/(x - y)V(1,2).
\end{align*}
Note that if we had taken the sum over all $i \ge 0$, we would get the same answer in the quotient; this is convenient when we don't explicitly know the order of the pole of the element we start with. 

With these explicit formulas to hand, it is easy to see that
\begin{align*}
\theta(\cdot)\vert_{ z= (x_1 - x_2); w=0} = \Delta_*(\alpha_*^V) \circ \gamma (\cdot). 
\end{align*}

From this it follows that $\lambda^\ch = \lambda_0^\ch$, and hence that $\mu^\ch = \mu_0^\ch$, as claimed. 
\end{proof}

\section{From factorization algebras on \texorpdfstring{$X$}{X} to \texorpdfstring{$(A, H, S_X)$}{(A,H,S)}-vertex algebras}\label{Section: from factorization to AHS}

Composing the functor $\Psi_X \circ \Phi_X$ with the equivalence from chiral algebras on $X$ to factorization algebras on $X$, we obtain a functor $F: \VA(A, H, S_X) \to \FA(X)$. In this section, we discuss a functor going in the opposite direction.

\begin{thm}\label{Theorem: Factorization to AHS}
Let $\CB = \{\CB_{X^I} \}$ be a unital factorization algebra on $X \subset \BBA^1$. For each finite set $I$, let $\BV(I) = \Gamma(X^I, \CB_{X^I})$; then $\BV= \{\BV(I)\}$ has a natural structure of $(A, H, S_X)$-vertex algebra. This extends to a functor $\Gamma_X$ from the category $\FA(X)$ of factorization algebras on $X$ to the category of $(A, H, S_X)$-vertex algebras. 
\end{thm}

\begin{proof}
Let us first define the structure of $\BV$ as a functor from $\Fin$ to the category of vector spaces. Given a morphism $\alpha: I \to J$ of finite sets, we define $\alpha^V_*: \BV(I) \to \BV(J)$ using the Ran condition data of $\CB$. Indeed, we have
\begin{align*}
\nu_\alpha: \phi(\alpha)^*\CB_{X^I} \to \CB_{X^J},
\end{align*}
which gives a map $\nu_\alpha: S_X(J) \otimes_{S_X(I)} \BV(I) \to \BV(J)$ on global sections. Then for $v \in \BV(I)$, we set $\alpha^V_* (v) = \nu_\alpha (1 \otimes v) \in \BV(J)$. The fact that the morphisms $\nu_\alpha$ are compatible with composition of the maps $\alpha$ implies that $\alpha^V_*$ are too. 

Next, $T(I)$ and $S_X(I)$ act naturally on $\BV(I)$, because $\CB_{X^I}$ is a $\CD_{X^I}$-module, and $T(I), S_X(I) \subset \Gamma(X^I, \CD_{X^I})$. Furthermore, because the map $\nu_\alpha$ is compatible with the $\CD$-module structures of $\CB_{X^I}$ and $\CB_{X^J}$, the $T$- and $S_X$-actions on $\BV(I)$ and $\BV(J)$ are intertwined by $\alpha^V_*$ in the appropriate way. 

The singular multiplication $\mu: \BV\odot \BV \to \BV$ is induced by the factorization isomorphisms in the following way. Recall that to define $\mu$ we need to give for each pair of finite sets $I_1, I_2$ a map
\begin{align*}
\mu_{I_1,I_2}: (\BV(I_1) \otimes_\BC \BV(I_2)) \otimes_{S_X(I_1, I_2)} S_X(I_1: I_2) \to \BV(I_1: I_2).
\end{align*}
This map should be compatible with the actions of $S_X(I_1: I_2)$ and $T(I_1: I_2)$ on each side, and furthermore the collection of maps $\mu_{I_1, I_2}$ should be functorial in $I_1$ and $I_2$. 

So let us choose a pair $I_1, I_2$. Let $I = I_1 \sqcup I_2$, and let $\alpha: I \to \{1,2\}$ be the obvious map. Then $\mu_{I_1, I_2}$ is defined to be the following composition:
\begin{align*}
(\BV(I_1) \otimes \BV(I_2)) \otimes_{S_X(I_1, I_2)} S_X(I_1: I_2) = \Gamma(X^I, j(\alpha)_* j(\alpha)^* (\CB_{X^{I_1}} \boxtimes \CB_{X^{I_2}})) \\
 \EquivTo \Gamma(X^J, j(\alpha)_* j(\alpha)^* \CB_{X^I}) = \BV(I_1: I_2).
\end{align*} 
That is, $\mu_{I_1, I_2} (v_1 \otimes v_2) = d_\alpha (v_1 \otimes v_2)$. Because $d_\alpha$ is a morphism  of $\CD$-modules, $\mu_{I_1, I_2}$ is compatible with the actions of $S_X$ and $T$. 

To check the functoriality of $\mu$, we consider sets $K_1$ and $K_2$ with maps $\gamma_i: K_i \to I_i$, inducing $\gamma: K = K_1 \sqcup K_2 \to I$. We need to show that the following diagram commutes:

\begin{center}
\begin{tikzpicture}[>=angle 90] 
\matrix(b)[matrix of math nodes, row sep=2em, column sep=2em, text height=1.5ex, text depth=0.25ex]
{\BV(K_1) \otimes \BV (K_2) & \BV(K_1: K_2)\\
  \BV(I_1) \otimes \BV (I_2) & \BV(I_1: I_2). \\};
\path[->, font=\scriptsize]
 (b-1-1) edge node[above]{$\mu_{K_1, K_2}$} (b-1-2)
 (b-1-2) edge node[right]{$\gamma^V_*$} (b-2-2)
 (b-1-1) edge node[left]{$\gamma_{1,*}^V \otimes \gamma_{2,*}^V$} (b-2-1)
 (b-2-1) edge node[below]{$\mu_{I_1, I_2}$}(b-2-2);
\end{tikzpicture}
\end{center}
This follows immediately from the compatibility between the $\nu$s and the $d$s. 

It remains to check that the resulting singular multiplication map is associative and commutative. Both properties follow from the compatibility of the morphisms $d$ under composition. 

Finally, to see that $\BV$ is a unital $(A, H, S_X)$-vertex algebra, observe (similarly to the arguments appearing in Section 3.4.5 of \cite{BD}) that the factorization isomorphism for the decomposition $\varnothing= \varnothing \cup \varnothing$ makes $\BV(\varnothing)$ into an algebra with multiplication $\mu: \BV(\varnothing) \otimes \BV(\varnothing) \to \BV(\varnothing)$ an isomorphism of vector spaces. We conclude that $\BV(\varnothing) = \BC$, with unit $1$. 
\end{proof}

\section{Composition of functors}\label{sec: composition of functors}
We have now constructed the following (non-commutative) diagram of functors, where the horizontal arrows in the top row are the equivalences of Theorem \ref{Theorem: everybody else's theorem}:

\begin{center}
\begin{tikzpicture}[>=angle 90, bij/.style={above,inner sep=0.5pt}]
\matrix(b)[matrix of math nodes, row sep=2em, column sep=2em, text height=1.5ex, text depth=0.25ex]
{\VA(X) & \CAlg(X) & \FA(X) \\
& \VA(A, H, S_X)& \\};
\path[->, font=\scriptsize]
 (b-1-1) edge node[bij]{$\sim$} node[below]{$\Psi_X$} (b-1-2)
 (b-1-2) edge node[bij]{$\sim$} (b-1-3)
 (b-2-2) edge node[below left]{$\Phi_X$} (b-1-1)
 (b-1-3) edge node[below right]{$\Gamma_X$} (b-2-2);
\end{tikzpicture}
\end{center}

Let us fix some notation. Suppose that $(V, Y(\bullet, z), \vac)$ is a vertex algebra on $X$; let $\CV$ denote the underlying left $\CD_X$-module, and let $\CV^r$ denote the corresponding right $\CD_X$-module. The chiral algebra corresponding to the vertex algebra has underlying right $\CD_X$-module $\CV^r$, and chiral bracket $\mu^\ch$ corresponding to a map $\lambda^\ch$ of left $\CD_X$-modules as in (\ref{Eq: chiral bracket}). Finally, let us denote the  corresponding factorization algebra by $\CB = \{ \CB_{X^I} \}$. Unwinding the definitions of the functors in Theorem \ref{Theorem: everybody else's theorem}, one sees that $\CB$ has the following properties: 
\begin{itemize}
\item $\CB_X = \CV$ and $\CB_{X^2} = \Ker(\lambda^\ch: j_* j^* (\CV \boxtimes \CV) \to \Delta_* \CV)$.
\item The factorization isomorphism $d:  j^*(\CB_X \boxtimes \CB_X) \EquivTo j*(\CB_{X^2})$ is the restriction of the obvious inclusion $\CV \boxtimes \CV \emb j_* j^* (\CV \boxtimes \CV)$ to $U = X \times X \setminus \Delta$.
\item Ran's isomorphism $\nu: \CB_X \EquivTo \Delta^*\CB_{X^2} = \CB_{X^2} /(x_1-x_2)\CB_{X^2} $ is given on sections by $v \mapsto 1 \otimes v \mod (x_1-x_2)\CB_{X^2} $ (it follows from properties of the unit that $1 \otimes v \in \Ker(\lambda^\ch) \subset j_* j^*(\CV \boxtimes \CV)$). 
\item The inverse $\nu^{-1}$ can also be written explicitly as follows: given a section $A$ of $ \CB_{X^2} $, calculate $\lambda^\ch(A\cdot (x_1-x_2)^{-1})$, a section of $\Delta_* \CB_{X}$. Since $A$ is in the kernel of $\lambda^\ch$, we must obtain something of the form $\sum_{i} f_i(x_1) \otimes v_i \cdot (x_1-x_2)^{-1} = 1 \otimes \left(\sum_i f_i(x_2)\cdot v_i \right) (x_1-x_2)^{-1}$ (possibly $0$). Then it is easy to see that $A - \sum_i f_i(x_2)\cdot v_i$ are sections of the ideal sheaf $(x_1-x_2)\CB_{X^2} $, and so $A = \nu (\sum_i f_i(x) \cdot v_i)$ in $\Delta^*\CB_{X^2} $. 
\end{itemize}

\begin{thm}\label{Theorem: composition}
The functor $\Gamma_X$ provides a left-inverse to the functor $\Phi_X$ in the following sense:
\begin{align*}
\Phi_X \circ \Gamma_X (\CB) = V.
\end{align*}
\end{thm}

\begin{rmk}\label{Remark: not an inverse}
We will see in Section \ref{sec: failure to be an equivalence} that $\Phi_X$ is not an equivalence; in particular, $\Gamma_X$ does not provide a right-inverse to $\Phi_X$. 
\end{rmk}

\begin{proof}[Proof of theorem]
Let $\{V(I), \mu\}$ denote the $(A, H, S_X)$ vertex algebra $\Gamma_X (\CB)$. In particular, $V(1) = \Gamma(X, \CV) = V$, and so the vector space underlying the vertex algebra $\Phi_X \circ \Gamma_X(\CB)$ is indeed $V$, with the same $\Gamma(X, \CO_X)$-module structure as that of the original vertex algebra. Likewise, the derivation agrees with the original one. Furthermore, the unit morphism $\CO \to \CB$ of the factorization algebra is determined by the fact that the global section $1$ of $\CO_X$ is sent to $\vac \in V = \Gamma(X, \CB_X)$; hence the unit morphism $\{S_X(I) \to V(I) \}$ has the same property, and in particular the vector $\vac$ is still the vacuum vector of the vertex algebra $\Phi_X \circ \Gamma_X(\CB)$.

Let $\widetilde{Y}(\bullet, z)$ denote the vertex operator function of the new vertex algebra. It remains to show that for any $v, u \in V$, 
\begin{align*}
\widetilde{Y}(v, z) u = Y(v,z) u.
\end{align*}
Recall the definition (\ref{Eq: Y}) of $\widetilde{Y}$ in terms of the singular multiplication $\mu$: we must choose $N$ so that $\mu_{12}(v \otimes u) \cdot (x_1-x_2)^N = A $ is in $V(1,2) = \Gamma(X, \CB_{X^2})$; then 
\begin{align*}
\widetilde{Y}(v,z_1)u = \sum_{i\ge 0} \alpha^V_*(\del_1^{(i)} A) \otimes z_1^{i-N}.
\end{align*}
But both $\mu_{12}$ and $\alpha^V_*$ can be written in terms of the factorization structure of $\CB$: 
\begin{itemize}
\item $\mu_{12}(v \otimes u) = d(v \otimes u)$; of course this is just $v \otimes u$ as a section of $j_*j^*(\CV \boxtimes \CV)$, but when we think of it as a section of $j_*j^*(\CB_{X^2})$, we use the fact that by our assumption above, it is of the form $A \cdot (x_1-x_2)^{-N}$, where $A$ is a global section of $\CB_{X^2}$. Thus we conclude that $A = v\otimes u \cdot (x_1-x_2)^{N}$.
\item We factor $\alpha^V_*$ as follows:
\begin{align*}
V(1,2) \surj V(1,2) / (x_1-x_2)V(1,2) \xrightarrow{\nu^-1} V(1);
\end{align*}
in other words, $\alpha^V_*$ is calculated using the procedure described just before the statement of the theorem. 
\end{itemize}
Comparing with $Y(v,z)u = \sum_{n \in \BZ} v_{n} (u) \cdot z^{-n-1}$, we see that to prove the theorem, it is enough to show that for every $i \ge 0$, 
\begin{align}\label{Eq: equation for coefficients of vertex operator}
v_{N - i - 1}(u) = \alpha^V_*(\del_1^{(i)} A). 
\end{align}
(Note that the choice of $N$ ensures that $v_{N - i -1}(u) = 0$ for any $i < 0$: indeed, we know that $A \in \Ker(\lambda^\ch)$, or equivalently that $Y(v, (z-w))u \cdot (z-w)^N = \sum_{n \in \BZ} v_{n} u \cdot (z-w)^{N-n-1}$ has no non-zero non-positive powers of $(z-w)$.)

We claim that the following equation holds for any pair $v \otimes u$ together with an integer $N$ such that $v \otimes u \cdot (x-y)^N = A$ lies in $V(1,2)$, and any integers $i, k \ge 0$:
\begin{align}\label{Eq: induction claim}
\lambda^\ch (\del_1^i A \cdot (x_1-x_2)^{-k}) = \sum_{j = N - k - i}^{N - i - 1} \frac{(N - j - 1)!}{(N - j - 1 - i)!} v_{(j)}u \cdot (x_1-x_2)^{N - j-k-i -1}
\end{align}

This claim follows from a straightforward induction argument, using the definition of $A, N$, and $\lambda^\ch$ for the base case $i=0$, and the fact that $\lambda^\ch$ commutes with the action of $\del_1$ for the induction step.

In particular, consider the case $k = 1$: we conclude that 
\begin{align*}
\lambda^\ch(\del_1^i A \cdot (x_1-x_2)^{-1}) & = \frac{(N - (N - 1 - i) - 1)!}{(N - (N - 1 - i) - 1 - i)!} v_{N - 1 - i}u \cdot (x_1-x_2)^{N - (N - 1 - i) - 1 - i -1} \\
& = i! v_{N - 1 - i} u \cdot (x_1-x_2)^{-1}.\\
\end{align*}
This tells us that $\alpha^V_* (\del^i A ) =i! v_{N - 1 -i}u$; in other words, Equation (\ref{Eq: equation for coefficients of vertex operator}) holds. The proof is complete. 

\end{proof}

\section{Translation-equivariant version of the story}\label{Section: translation-equivariant}
In this section, we recall the notion of a translation-equivariant vertex algebra on $\BBA^1$. Then we introduce the notion of a translation-equivariant $(A,H,S_{\BBA^1})$-vertex algebra, and we show that the functor from $(A,H,S_{\BBA^1})$-vertex algebras to vertex algebras can be enhanced to a functor of the translation-equivariant categories; likewise, the functor from factorization algebras to $(A, H, S_{\BBA^1})$-vertex algebras can be enhanced to the translation-equivariant categories. We also give a functor from the category $(A,H,S_B)$-vertex algebras to the category of translation-equivariant $(A, H, S_{\BBA^1})$-vertex algebras. (Unlike the analogous functor from vertex algebras to translation-equivariant vertex algebras on $\BBA^1$, this functor is not an equivalence.)

\begin{defn}\label{def: translation-equivariant vertex algebras}
Recall that a vertex algebra $V$ on $\BBA^1$ is in particular a $\BC[x]$-module. We consider the map $p: \BBA^1 \times \BBA^1 \to \BBA^1$ given by projection onto the second factor, and also the map $a: \BBA^1 \times \BBA^1 \to \BBA^1$ given by addition (i.e. translation). The vertex algebra $V$ is said to be \emph{translation-equivariant} if it is equipped with an isomorphism $\psi: p^*V \to a^*V$ of $\BC[t, x]$-modules which is compatible with the vertex algebra structure. (More precisely, one can formulate the notion of a \emph{family} of vertex algebras over $\BBA^1$; $p^*V$ and $a^*V$ are examples of such families, and $\psi$ is an isomorphism between them.)

In particular, for any choice of $t=t_0$, we consider the restriction $V_{t_0}$ of $a^*V$ to $\{t_0\} \times \BBA^1$. As a vector space, $V_{t_0} = V$, but the $\BC[x]$ action is shifted by $x \mapsto x - t_0$. Under this action, $V_{t_0}$ is a vertex algebra over $\BBA^1$, with the same vacuum vector and vertex operators as $V$. The isomorphism $\psi: p^* V \to a^*V$ restricts to give an isomorphism $\psi_{t_0}: V \to V_{t_0}$ of vertex algebras over $\BBA^1$; these isomorphisms are compatible under composition and addition of the parameters $t_0$. 

Let us denote the category of translation-equivariant vertex algebras on $\BBA^1$ by $VA(\BBA^1)^{\BBA^1}$. 
\end{defn}

Recall also part two of Theorem \ref{Theorem: everybody else's theorem}, which states that the category of translation-equivariant vertex algebras on $\BBA^1$ is equivalent to the category of vertex algebras. This equivalence is given on the level of vector spaces by taking invariant vectors, or as explained in the appendix of \cite{BDHK}, by taking the stalk at $0$. 

Let us now formulate the definition of a \emph{translation-equivariant $(A,H,S)$-vertex algebra on $\BBA^1$}. First recall the definition of a family of $(A,H,S)$-vertex algebras over $\BBA^1$, as defined in Remark \ref{rmk: AHS in families}; here we are interested in the case $p_1: X = \BBA^1 \times \BBA^1 \to \BBA^1$. Observe that the second projection map $p:X  \to \BBA^1$ extends to projection maps $p^I: X^I_{\BBA^1} \to \BBA^I$ which are compatible with diagonal embeddings and projections. It follows that for an $(A, H, S_{\BBA^1})$-vertex algebra $V$, we can use pullback along $p$ to define a family of $(A, H, S)$-vertex algebras $p^*V$ over $\BBA^1$: for $I \in \Fin$, we have $p^*V(I) \defeq (p^I)^*(V(I)) = \BC[t, x_i]_{i \in I} \otimes_ {\BC[x_i]} V(I)$. (This is the \emph{trivial} family of $(A, H, S)$-vertex algebras.) Similarly, the action of $\BBA^1$ on itself by addition induces compatible maps $a: X \to \BBA^1$ and $a^I: X^I_{\BBA^1} \to \BBA^I$ (the latter can be identified with the diagonal action of $\BBA^1$ on $\BBA^I$); hence we obtain a second family $a^* V$. 

\begin{defn}
A \emph{translation-equivariant $(A,H,S_{\BBA^1})$-vertex algebra} is a pair $(V, \psi)$, where $V$ is an $(A,H,S_{\BBA^1})$-vertex algebra, and $\psi: p^*V \to a^*V$ is an isomorphism of families of $(A,H,S)$-vertex algebras. The isomorphism $\psi$ is required to satisfy a natural compatibility condition on $\BBA^1 \times \BBA^1 \times \BBA^1$: namely, letting $p_3: (x,y,z) \mapsto z$ and $a_3: (x,y,z) \mapsto x+y+z$, we obtain from $\psi$ two isomorphisms $p_3^* V \to a^* V$,
\begin{align}\label{eq: cocycle condition}
p_3^*V = (p \times \id)^* p^* V \EquivTo(p \times \id)^* a^* V = (\id \times a)^*p^*V \EquivTo (\id \times a)^* a^* V = a_3^* V;  \nonumber\\
p_3^*V = (a \times \id)^*p^* V \EquivTo (a \times \id)^* a^* V = a_3^*V.
\end{align}
The condition on $\psi$ is that these two compositions must be equal.

Let us denote the category of translation-equivariant $(A, H, S_{\BBA^1})$-vertex algebras by $\VA(A, H, S_{\BBA^1})^{\BBA^1}$. 
\end{defn}

Now let $\CB$ be any factorization algebra on $\BBA^1$. Similarly to the above, the families of maps $\{p^I\}$ and $\{a^I\}$ induced by $p$ and $a$, respectively, allow us to pull $\BBA^1$ back to get relative factorization algebras $p^*\CB$ and $a^*\CB$ on $\BBA^1 \times \BBA^1$ over $\BBA^1$ (with respect to the first projection map). 

\begin{defn}
A \emph{translation-equivariant factorization algebra on $\BBA^1$} is a factorization algebra $\CB$ on $\BBA^1$ together with an isomorphism $\psi_\CB: p^* \CB \to a^* \CB$, satisfying a compatibility condition analogous to the condition described in (\ref{eq: cocycle condition}) above.
\end{defn}

Let us denote the category of translation-equivariant factorization algebras on $\BBA^1$ by $\FA(\BBA^1)^{\BBA^1}$.

The following is immediate from the definitions. 
\begin{prop}\label{Prop: translation-equivariant factorization algebra to AHS}
The functor $\Gamma_{\BBA^1}$ naturally extends to a functor (still called $\Gamma_{\BBA^1}$)
\begin{align*}
\FA(\BBA^1)^{\BBA^1} \to \VA(A,H,S)^{\BBA^1}.
\end{align*}
\end{prop}

Recall the functor  $\Phi_{\BBA^1} : \VA(A, H, S_{\BBA^1}) \to \VA({\BBA^1})$ of Theorem \ref{Theorem: AHS to vertex algebra}. 

\begin{prop}\label{Prop: translation-equivariant AHS to vertex algebra}
The functor $\Phi_{\BBA^1}$ extends naturally to a functor (still denoted by $\Phi_{\BBA^1}$) 
\begin{align*}
\VA(A, H, S_{\BBA^1})^{\BBA^1} \to \VA({\BBA^1})^{\BBA^1}.
\end{align*}
\end{prop}
\begin{proof}
Given a pair $(V, \psi)$ in $\VA(A, H, S_{\BBA^1})^{\BBA^1}$, we know that $V(1)$ has a structure of vertex algebra on $\BBA^1$. The isomorphism $\psi: p^* V \to a^*V$ gives in particular an isomorphism of vector spaces $p^*(V(1)) \to a^*(V(1))$. It is not hard to check from the construction of $\Phi_{\BBA^1}$ that this isomorphism of vector spaces respects the vertex algebra structure. 
\end{proof}

\begin{prop}\label{Prop: Sb to translation equivariant by tensoring}
There is a natural functor  $\Xi: \VA(A,H,S_B) \to \VA(A, H, S_{\BBA^1})^{\BBA^1}$. 
\end{prop}
\begin{proof}
Given $V \in \VA(A,H,S_B)$, we define an $(A, H, S_{\BBA^1})$-vertex algebra $W$ by setting, for $I \in \Fin$, $W(I) = V(I) \otimes_{S_B(I)} S_{\BBA^1}(I) = V(I) \otimes_{\BC} \BC[x_i]_{i \in I}$. This is naturally a module over $S_B$, and we define the action of $T(I) = \BC[\del_i]$ by letting $\del_i$ act by $\del_i \otimes 1 + 1 \otimes \frac{d}{dx_i}$. It is easy to see that this gives an object of $\Fun(\Fin, A, T, S_{\BBA^1})$; one can also extend the singular multiplication of $V$ $\BC[x_i]$-linearly to give a singular multiplication on $W$. 

It remains to show that the resulting $(A, H, S_{\BBA^1})$-vertex algebra $W$ can be given an equivariant structure $\psi$. For any $I \in \Fin$, we need to define $\psi^I: p^* W(I) \EquivTo a^* W(I)$ . Unwinding the defintions, we can see that $p^*W(I) \cong \BC[t, x_i] \otimes V(I)$, and likewise $a^*W(I) \cong \BC[t, x_i] \otimes V(I)$; we can take $\psi^I$ to be the identity map. It is immediate that this is functorial and respects the singular multiplication structures.
\end{proof}

\begin{rmk}[Remarks on Proposition \ref{Prop: Sb to translation equivariant by tensoring}]
Note that the $(A,H,S_{\BBA^1})$-vertex algebra $W$ constructed in the proof is almost never in the image of the functor $\Gamma_{\BBA^1} : \FA(\BBA^1) \to \VA(A, H, S_{\BBA^1})$. Indeed, we can see that factorization fails to hold on $\BBA^2$ as follows. Let $\CW(I)$ be the sheaf on $\BBA^I$ associated to $W(I)$. We can see that the restriction of $\CW(1,2)$ to the diagonal has global sections $V(1,2) \otimes \BC[x]$, while its restriction to the complement of the diagonal has global sections $V(1,2) \otimes \BC[x_1,x_2, (x_1-x_2)^{-1}]$. If we assume that $\CW$ is a factorization algebra, then from the first fact, we conclude that $V(1,2) \cong V(1)$, while from the second, we see that $V(1,2) \cong V(1) \otimes V(1)$; hence we must have $V(1) \otimes V(1) \cong V(1)$. 
\end{rmk}
\section{The failure of Borcherds's functor to be an equivalence: case study for the Virasoro and the lattice VOA.}\label{sec: failure to be an equivalence}
In this section, we will show that the functor $\Phi: \VA(A, H, \Sb) \to \VA$ fails to be an equivalence. More precisely, recall Example \ref{Example: conformal structure on lattice VOA}, in which we defined the \emph{conformal structure} on the rank 1 lattice vertex algebra $V^L$. Recall also Example \ref{Example: lattice AHS}: Borcherds constructs an example $\BV^L$ of an $(A, H, \Sb)$-vertex algebra such that $\Phi(\BV^L)$ is the lattice vertex algebra $V_L$. If $\Phi$ were essentially surjective, we could find an object $W \in \VA(A, H, \Sb)$ such that $\Phi(W) \cong \Vir_1$; and then if $\Phi$ were in addition fully faithful, we could also find a morphism $F: W \to \BV^L$ such that $\Phi(F)$ recovered the map of vertex algebras described in Example \ref{Example: conformal structure on lattice VOA},
\begin{align*}
\phi_0: \Vir_1 & \to V_L\\
L_{-2}v_1 & \mapsto \omega_0.
\end{align*}

We will prove that no such pair $(W, F)$ exists, and thus will have proved the following theorem:
\begin{thm}\label{Theorem: Borcherds doesn't give an equivalence, actual}
The functor $\Phi: \VA(A, H, \Sb) \to \VA$ is not an equivalence of categories. 
\end{thm}

\begin{proof}
Suppose towards a contradiction that we have a pair $(W, F)$ as described above. 

In particular, we can identify $W(1)$ with $\Vir_1$, and $F(1)$ with the map $\phi_0 : \Vir_1 \to V_L \cong V^L$. There is an element $\omega \in W1)$ corresponding to $L_{-2}v_1 \in \Vir_1$; its image in $V^L$ under $F(1)$ is the image of $\omega_0 = \frac{1}{2} b_{-1}^2\vac \in V_L$ under the identification $f$ of (\ref{Eq: identifying V^L and V_L}). Notice that $\frac{1}{2} b_{-1}^2 \vac = \frac{1}{2}(b_{-1}\vac)^2$ in the commutative algebra structure of $V_L$, and $b_{-1}| 0 \rangle = b_{-1}|\lambda_0 \rangle \cdot | - \lambda_0 \rangle$ (here $\lambda _0 = \sqrt{N}$ is the generator of $L = \sqrt{N}\BZ$, and $b = \sqrt{N} \otimes \frac{1}{\sqrt{N}}$ is the generator of $\fh$). Since $b_{-1} | \lambda _ 0 \rangle = \frac{1}{\sqrt{N}} (\lambda_0 )_{-1}| \lambda_0 \rangle = \frac{1}{\sqrt{N}} T | \lambda_0 \rangle$, we conclude from the properties of $f$ that $F(\pt)(\omega) = \frac{1}{2 N} e^{-2 \lambda_0} (\lambda_0(1))^2$. Let us denote this element by $\alpha$. 

Let us now outline the strategy of the proof, before continuing with the computations. 

\begin{description}
\item[Step 1] We compute $\mu_{V^L}(\alpha \otimes \alpha)$ directly, and obtain
\begin{align}\label{Eq: alpha computation}
\alpha \otimes \alpha + \frac{1}{N}e^{-\lo} \lo(1) \otimes e^{\lo} \lo(1) (x-y)^2 + \frac{1}{2} e^0 \otimes e^0 (x-y)^{-4}.
\end{align}
\item[Step 2] Using the compatibility of $F$ with the multiplication structures on $W$ and $\BV^L$, we see that 
\begin{align*}
F(1:2) \mu_W(\omega \otimes \omega) = \mu_{V^L}(\alpha \otimes \alpha);
\end{align*}
combining this with the equation (\ref{Eq: alpha computation}) and the fact that $F(1:2) = F(1,2) \otimes \id_{\BC[(x-y)^{\pm1}]}$, we conclude that there exists some element $\beta \in W(1,2)$ such that $F(1,2)(\beta) = \alpha \otimes \alpha$. 
\item[Step 3] We consider $\del_x \beta$; by compatibility of $F$ with the functoriality of $W$ and $\BV^L$ with respect to the map $\gamma: \{1,2\} \to \{1\}$ we have that 
\begin{align}\label{Eq: beta equation}
\phi_0 \circ \gamma^W_* (\del_x \beta) = f\circ \gamma^{\BV^L}_* \circ F(1,2) (\del_x \beta).
\end{align}
However, we compute the right-hand side of this equation directly, and show that it is not in the image of $\phi_0$. This contradicts the assumption that the pair $(W, F)$ exists, and hence proves the theorem. 
\end{description}

To complete the proof, we will now explain each step. 

\subsection{Step 1: Computation of \texorpdfstring{$\mu_{V^L}(\alpha \otimes \alpha)$}{multiplication of alpha with itself}}
Recall from (\ref{Eq: definition of mu}) that $\mu_{V^L}(\alpha \otimes \alpha) = (\alpha_{(1)} \otimes \alpha_{(1)}) r(\alpha_{(2)} \otimes \alpha_{(2)})$. We compute that 
\begin{align*}
\Delta(\alpha) 
  & = \frac{1}{2N} \left(e^{-2\lambda_0} (\lambda_0(1))^2 \otimes e^0 + 2 e^{-\lambda_0} \lambda_0 (1) \otimes e^{-\lambda_0} \lambda_0 (1) + e^0 \otimes e^{-2\lambda_0} (\lambda_0(1))^2\right).
\end{align*}
In order to compute $\mu_{V^L}(\alpha \otimes \alpha)$, we need to compute $r(\alpha_{(2)}\otimes\alpha_{(2)})$ for each of the three terms in $\alpha_{(2)}$. 

Before we begin our computation, let us write down the value of $r$ on a few basic pairs of elements. This will help us get used to working with $r$, and will also be useful in later calculations. 

First, from the definition of $r$, we have
\begin{align}\label{Eq: r on pure}
r(e^{n\lo} \otimes e^{m\lo}) = (x-y)^ {nmN}.
\end{align}

In particular, when $n = m =0$,
\begin{align}\label{Eq: r on trivial}
r(e^0 \otimes e^0) = 1.
\end{align}

The compatibility of $r$ with the action of $\del$ on the first and second factors, respectively, gives us the following:
\begin{align}
r(\lo(1) \otimes e^{m\lo}) &= \del_x r(e^\lo \otimes e^{m\lo}) = mN(x-y)^{mN -1} \label{Eq: T on x}\\
r(e^{n\lo} \otimes \lo(1)) & = -nN(x-y)^{nN -1} \label{Eq: T on y}.
\end{align}

In particular, when $m$ or $n$ are zero, respectively, the value of $r$ is 0. Combining, we see that 
\begin{align}
r(\lo(1) \otimes \lo(1)) &= N(N-1) (x-y)^{N -2} \label{Eq: T on both}.
\end{align}

To calculate more complicated instances of $r$, we use the compatibility of $r$ with the coproduct $\Delta$. Let us record the following values of $\Delta$, for future reference:
\begin{align}
\Delta(e^{n\lo}) &= e^{n\lo} \otimes e^{n\lo} \label{Eq: Delta on group-like};\\
\Delta(\lo(1)) & = \lo(1) \otimes e^\lo + e^\lo \otimes \lo(1) \label{Eq: Delta on primitive};\\
\Delta((\lo(1)^2) &= (\lo(1))^2 \otimes e^{2\lo} + 2 e^\lo  \lo(1) \otimes e^\lo  \lo(1) + e^{2\lo} \otimes (\lo(1))^2 \label{Eq: Delta on primitive squared}.
\end{align}

Now we begin the calculation of $\mu_{V^L}(\alpha \otimes \alpha)$. First, from (\ref{Eq: r on trivial}), we have 
\begin{align*}
r(e^0 \otimes e^0) = 1.
\end{align*}

Next, we need to compute 

\begin{align*}
r\left(e^{-\lambda_0}\lambda_0(1) \otimes e^{-\lambda_0}\lambda_0(1)\right) & = r\left(e^{-\lambda_0} \otimes \left(e^{-\lambda_0}\lambda_0(1)\right)_{(1)}\right)\cdot r\left(\lambda_0(1) \otimes \left( e^{-\lambda_0}\lambda_0(1)\right) _{(2)}\right).
\end{align*}
We have
\begin{align*}
\Delta\left( e^{-\lambda_0}\lambda_0(1) \right) 
 & = e^{-\lo} \lo(1) \otimes e^0 + e^0 \otimes e^{-\lo} \lo(1),
\end{align*}
so $r\left(e^{-\lambda_0}\lambda_0(1) \otimes e^{-\lambda_0}\lambda_0(1)\right)$ is a sum of two terms, 
\begin{align}
r(e^{-\lo} \otimes e^{-\lambda_0}\lambda_0(1)) \cdot r(\lo(1) \otimes e^0) \label{Eq: first} \\ 
+ \quad r(e^{-\lo} \otimes e^0) \cdot r(\lo(1) \otimes e^{-\lambda_0}\lambda_0(1)). \label{Eq: second}
\end{align} 

By (\ref{Eq: T on x}), the term (\ref{Eq: first}) is 0.

 To compute (\ref{Eq: second}), we observe that $r(e^{-\lo} \otimes e^0) = 1$, and that
\begin{align*}
r(\lo (1) \otimes e^{-\lo} \lo(1)) & = \del_x \left( r(e^\lo \otimes e^{-\lo}\lo(1)) \right) = \del_x \left( r(e^\lo \otimes e^{-\lo}) \cdot r(e^\lo \otimes \lo(1) )\right)\\ 
 & = \del_x \left[ (x-y)^{-N} \cdot (-N)(x-y)^{N-1} \right] 
  = N(x-y)^{-2}.
\end{align*}

We conclude that 
\begin{align*}
r\left(e^{-\lambda_0}\lambda_0(1) \otimes e^{-\lambda_0}\lambda_0(1)\right) = N (x-y)^{-2}.
\end{align*}

Finally we need to calculate $R \defeq r(e^{-2\lo} (\lo(1))^2 \otimes e^{-2\lo} (\lo(1))^2)$. Using our calculation of $\Delta (e^{-\lo} \lo(1))$ from above, we write down $\Delta(e^{-2\lo} (\lo(1))^2)$; then we can write $R$ as a sum of three terms as follows:
\begin{align}
\label{Eq: a} r(e^{-2\lo} \otimes e^{-2\lo} (\lo(1))^2) \cdot r( (\lo(1))^2 \otimes e^0) \\
\label{Eq: b} + \quad 2 r(e^{-2\lo} \otimes e^{-\lo} \lo(1)) \cdot r( (\lo(1))^2 \otimes e^{-\lo} \lo(1) )\\
\label{Eq: c} + \quad r(e^{-2\lo} \otimes e^0) \cdot r ((\lo(1))^2 \otimes e^{-2\lo} (\lo(1))^2) .
\end{align}
However, we check that $r( (\lo(1))^2 \otimes e^0)$ and $r( (\lo(1))^2 \otimes e^{-\lo} \lo(1) )$ are both $0$, so the terms (\ref{Eq: a}) and (\ref{Eq: b}) vanish. To calculate (\ref{Eq: c}), we note that $r(e^{-2\lo} \otimes e^0) = 1,$ and we use the expansion of $\Delta(\lo(1))^2$ to write $r ((\lo(1))^2 \otimes e^{-2\lo} (\lo(1))^2)$ as a sum
\begin{align}
\label{Eq: d} r(e^{2\lo} \otimes e^{-\lo}\lo(1)) \cdot r((\lo(1))^2 \otimes  e^{-\lo}\lo(1) ) \\
\label{Eq: e} +\quad  2r(e^\lo \lo(1) \otimes  e^{-\lo}\lo(1)) \cdot r( e^\lo \lo(1) \otimes e^{-\lo}\lo(1) \\
\label{Eq: f} +\quad r( (\lo(1))^2 \otimes e^{-2\lo} ) \cdot r( e^{2\lo} \otimes  e^{-\lo}\lo(1)). 
\end{align}
The first and last terms vanish, and so it remains to calculate the term (\ref{Eq: e}). Using the calculation of $\Delta(e^{-\lo} \lo(1))$ from above, we see that $r(e^\lo \lo(1) \otimes e^{-\lo} \lo(1))$ is equal to
\begin{align}
 r(e^\lo \otimes e^{-\lo} \lo(1)) \cdot r(\lo(1) \otimes e^0) \label{Eq: g}\\
 + r(e^\lo \otimes e^0) \cdot r(\lo(1) \otimes  e^{-\lo} \lo(1)) \label{Eq: h}.
\end{align}
Once again, (\ref{Eq: g}) is 0, and $r(e^\lo \otimes e^0) = 1$, so 
\begin{align*}
r(e^\lo \lo(1) \otimes e^{-\lo} \lo(1)) =  r(\lo(1) \otimes  e^{-\lo} \lo(1)) = N(x-y)^{-2}. 
\end{align*}
From this it follows that 
\begin{align*}
r ((\lo(1))^2 \otimes e^{-2\lo} (\lo(1))^2) &= 2 r(e^\lo \lo(1) \otimes e^{-\lo} \lo(1))^2 = 2N^2(x-y)^{-4}. 
\end{align*}
Combining our results, we conclude that
\begin{align*}
\mu_{V^L}(\alpha \otimes \alpha)
&= \frac{1}{4N^2}  e^{-2\lo} (\lo(1))^2 \otimes e^{-2\lo} (\lo(1))^2 \\
& \qquad + \frac{1}{N^2}  e^{-\lo}\lo(1) \otimes e^{-\lo} \lo(1) \cdot N (x-y)^{-2} \\
& \qquad + \frac{1}{4N^2} e^0 \otimes e^0 \cdot  2N^2(x-y)^{-4},
\end{align*}
and so
\begin{align}\label{Eq: conclusion}
\mu_{V^L}(\alpha \otimes \alpha) = \alpha \otimes \alpha + \frac{1}{N}e^{-\lo}\lo(1) \otimes e^{-\lo} \lo(1) \cdot (x-y)^{-2}  + \frac{1}{2} e^0 \otimes e^0 \cdot (x-y)^{-4},
\end{align}
as claimed. 

\subsection{Step 2: Existence of \texorpdfstring{$\beta$}{beta}}

Consider now the following commutative diagram, showing the compatibility of $F$ with the singular multiplication structures on $W$ and $\BV^L$:

\begin{center}
\begin{tikzpicture}[>=angle 90] 
\matrix(b)[matrix of math nodes, row sep=2em, column sep=3em, text height=1.5ex, text depth=0.25ex]
{W(1) \otimes W(2) \otimes \BC[(x - y) ^{\pm 1}] & W(1:2) = W(1, 2) \otimes  \BC[(x - y) ^{\pm 1}] \\
V^L \otimes V^L \otimes \BC[(x-y)^{\pm 1}] & V(1:2) = V^L \otimes V^L \otimes \BC[(x-y)^{\pm 1}].\\ };
\path[->, font=\scriptsize]
 (b-1-1) edge node[above]{$\mu_W$} (b-1-2)
 (b-2-1) edge node[below]{$\mu_{V^L}$} (b-2-2)
 (b-1-1) edge node[left]{$F(\pt) \otimes F(\pt) \otimes \id_{\BC[(x-y)^{\pm 1}}$} (b-2-1)
 (b-1-2) edge node[right]{$F(1:2) = F(1,2) \otimes \id_{\BC[(x-y)^{\pm 1}}$} (b-2-2);
\end{tikzpicture}
\end{center}

In particular, $\mu_W(\omega \otimes \omega) \in W(1:2)$ has the property that $F(1:2)(\mu_W(\omega \otimes \omega)) = \mu_{V^L} (\alpha \otimes \alpha)$. It follows that there is some element $\beta$ of $W(1,2)$ such that $F(1,2)(\beta) = \alpha \otimes \alpha$.

\subsection{Step 3: Computations with \texorpdfstring{$\del_x \beta$}{derivative of beta}}

Since $F$ is required to intertwine the functoriality of $W$ and $\BV^L$, the following diagram must commute:
\begin{center}
\begin{tikzpicture}[>=angle 90] 
\matrix(b)[matrix of math nodes, row sep=2em, column sep=3em, text height=1.5ex, text depth=0.25ex]
{W(1,2) & W(1) & \Vir_1\\
 \BV^L(1,2) & \BV^L(1) & \\
 V^L \otimes V^L  & V^L & V_L.\\};
 \path[->, font=\scriptsize]
 (b-1-1) edge node[above]{$\gamma^W_*$} (b-1-2)
 (b-2-1) edge node[above]{$\gamma^V_*$} (b-2-2)
 (b-3-1) edge node[above]{$m$} (b-3-2)
 (b-1-1) edge node[left]{$F(1,2)$} (b-2-1)
 (b-1-2) edge node[right]{$F(1)$} (b-2-2)
 (b-1-3) edge node[right]{$\phi_0$} (b-3-3)
 (b-3-2) edge node[above]{$f$} (b-3-3);
 \path[-]
 (b-2-1) edge[double] (b-3-1)
 (b-2-2) edge[double] (b-3-2)
 (b-1-2) edge[double] (b-1-3);
\end{tikzpicture}
\end{center}

In particular, we must have
\begin{align}\label{eq: commutativity condition}
f \circ m \circ F(1,2) (\del_x \beta) = \phi_0 \circ \gamma^W_* (\del_x \beta). 
\end{align}

We calculate the left-hand side as follows:
\begin{align*}
m \circ F(1,2) (\del_x \beta) &= m (\del_x (F(1,2)(\beta)) ) = m (\del_x (\alpha \otimes \alpha) ) = m ((\del_x \alpha) \otimes \alpha). 
\end{align*}

From the definition of $\alpha$, we calculate that 
\begin{align*}
\del_x (\alpha) 
 & = \frac{1}{N} \left( 2 e^{-2\lo} \lo(1) \lo(2) - e^{-3\lo}(\lo(1))^3 \right),
\end{align*}
and so 
\begin{align*}
m ((\del_x \alpha) \otimes \alpha) = \frac{1}{2N^2}\left(2e^{-4\lo}(\lo(1))^3\lo(2) - e^{-5\lo}(\lo(1))^5 \right).
\end{align*}

Under the isomorphism $f: V^L \to V_L$, this corresponds to the element
\begin{align*}
\frac{1}{2N^2} (\lo)_{-1}^3 (\lo)_{-2} \vac = \frac{1}{2}b_{-2} b_{-1}^3 \vac.
\end{align*}

We claim that this element is not in the image of $\phi_0$, contradicting the equation ($\ref{eq: commutativity condition}$), and thus proving the theorem. To prove this, consider the gradings on $\Vir_1$ and $V_L$ defined by setting
\begin{align*}
\deg v_1 = 0, \qquad & \deg L_{-n} = n \qquad \text{ for $\Vir_1$};\\
\deg b_{-n} = n  \qquad & \qquad  \text{ for $V_L$}. 
\end{align*}

Note that $\phi_0$ is compatible with these gradings. Note also that the degree 5 piece of $\Vir_1$ has dimension 2---it is spanned by the vectors $L_{-5}v_1$ and $L_{-2}L_{-3}v_1$ (since $L_{-3}L_{-2}v_1$ can be written as a linear combination of the latter two vectors). It follows that the image of $\phi_0$ in degree 5 has dimension at most two, and is spanned by $\phi_0(L_{-5}v_1)$ and $\phi_0(L_{-2}L_{-3}v_1)$.

We calculate directly from the definition of $\phi_0$ that 
\begin{align*}
\phi_0(L_{-5}v_1) & = b_{-4}b_{-1}\vac + b_{-3}b_{-2}\vac;\\
\phi_0(L_{-2}L_{-3}v_1) & = 2 b_{-4}b_{-1} \vac + b_{-3}b_{-2}\vac + \frac{1}{2}b_{-2}b_{-1}^3\vac.
\end{align*}

Since $\frac{1}{2}b_{-2}b_{-1}^3\vac$ is not in the span of these two vectors, it cannot be in the image of $\phi_0$. This completes Step 3, and the proof of the theorem. 

\end{proof}

\begin{rmk}\label{Remark: modify morphisms}
An interesting question is whether the notion of a morphism of $(A,H,S)$-vertex algebras can be weakened so that the properties of the morphism $F$ used to reach a contradiction in the proof of Theorem \ref{Theorem: Borcherds doesn't give an equivalence, actual} above no longer hold. One would hope that in that case, the pair $(W, F)$ could be defined; one would also hope that the lattice $(A,H,S)$-vertex algebra defined by Borcherds could be identified in $\VA(A,H,S_{\BBA^1})^{\BBA^1}$ with the $(A, H, S)$-vertex algebra coming from the translation-equivariant lattice factorization algebra.
\end{rmk}

\bibliographystyle{alpha}
\bibliography{AHS-bibliography}

\end{document}